\documentclass[12pt]{article}

\usepackage{amsmath}
\usepackage{amsfonts,amssymb,amsthm}
\usepackage{color}
\usepackage{graphicx}
\usepackage{listings}
\usepackage{hyperref}

\numberwithin{equation}{section}

\newtheorem{theorem}{Theorem}
\newtheorem{lemma}{Lemma}[section]
\newtheorem{prop}{Proposition}[section]
\newtheorem*{corollary}{Corollary}

\theoremstyle{definition}
\newtheorem*{remark}{Remark}
\newtheorem{definition}{Definition}[section]

\newcounter{appndx}
\renewcommand{\theappndx}{\Alph{appndx}}
\newcommand{\appndx}[2]{\refstepcounter{appndx}%
             \renewcommand{\thesection}{\theappndx}%
             \phantomsection\addcontentsline{toc}{section}{Appendix \theappndx: #1}
             \label{#2}\section*{Appendix \theappndx: #1} }

\def\ba#1{\begin{array}{#1}}
\def\ea{\end{array}}
\def\beq#1{\begin{equation}\label{#1}}
\def\eeq{\end{equation}}
\newcommand{\ab}[1]{\left\langle{#1}\right\rangle}

\newcommand{\dist}{\mathrm{dist}\,}
\newcommand{\dst}{\displaystyle}
\newcommand{\RR}{\mathbb{R}}
\newcommand{\NN}{\mathbb{N}}

\newcommand{\ZZ}{\mathbb{Z}}

\newcommand{\eps}{\varepsilon}
\newcommand{\talpha}{\tilde\alpha}
\newcommand{\tgamma}{\tilde\gamma}
\newcommand{\txi}{\tilde\xi}
\newcommand{\ali}{\alpha_{\infty}}
\newcommand{\Xmin}{X^{0}}
\newcommand{\Xcross}{X^{\!\times}}
\newcommand{\Tmin}{t^{0}}
\newcommand{\Tleft}{t^{\ell}}
\newcommand{\Tright}{t^{r}}
\newcommand{\eqbydef}{\stackrel{\text{\scriptsize\rm def}}{=}}

\newcommand{\Xb}{X^{**}}

\newcommand{\maplecodestyle}{\color{black}\scriptsize}
\newcommand{\mapleoutpstyle}{\color{blue}\scriptsize}

\usepackage{pgf,tikz,pgfplots}
\usepackage{mathrsfs}
\usetikzlibrary{arrows}

\title{Precise asymptotics with log-periodic term in an elementary optimization problem}

\author{Sergey Sadov\footnotemark[1]\footnote{E-mail: serge.sadov@gmail.com}}
\date{}

\begin{document}
\maketitle

\begin{abstract}
The function $\inf_n nx^{1/n}$ has the asymptotics $eu+e d^2(u)/(2u)+O(1/u^2)$ as $x\to\infty$, where
$u=\log x$ and $d(u)$ is the distance from $u$ to the nearest integer. We generalize this observation. 

First,
the curves $y=nx^{1/n}$ can be written parametrically
as $\log x=nt$, $y=nt$. In general, let 
$(u_n(t),v_n(t))$ be a family of parametric
curves  with asymptotics 
$u_n=n p_1(t)+q_1(t)+r_1(t)/n+O(1/n^2)$ and 
$v_n=n p_2(t)+q_2(t)+r_2(t)/n+O(1/n^2)$.
Suppose the function $p_1(t)/p_0(t)$ has a unique 
nondegenerate minimum in the parameter domain. 
It is shown that the asymptotics of their lower envelope
$v(u)=\inf_{n,t} v_n(t)$ while $u=u_n(t)$, has the asymptotics of the form $v(u)=a_0 u+a_1+\Phi(u)/u+O(1/u^2)$, where $\Phi(\cdot)$ is an affinely transformed function $d^2(\cdot)$.

Second, note that $nx^{1/n}$ is the minimum
of the sum $t_1+t_2/t_1+\dots+t_{n}/t_{n-1}$ subject to the constraint $t_n=x$. We consider a similar
asymptotic problem for the sums $t_1+t_2/(t_1+1)+\dots+t_n/(t_{n-1}+1)$. 
Let $F_n(x)$ is the minimum value of the $n$-term sum under the constraint $t_n=x$. Define $F(x)=\inf_n F_n(x)$. We show that $F(x)=eu-A+e d^2(u+b)/(2u)+O(1/u^2)$, $u=\log x$, with certain numerical constants $A$ and $b$. We present alternative forms of this optimization problem, in particular, 
a ``least action'' formulation. Also we find the asymptotics $F_n^{(p)}(x)=e\log n-A(p)+O(1/\log n)$
for the function arising from the sums with denominators of the form $t_j+p$ with arbitrary $p>0$ and establish
some facts about the function $A(p)$. 


\end{abstract}

\section{Statement of results}
\label{sec:intro}

The main result of this work concerns the asymptotic behaviour as $x\to+\infty$ of the function
$$
 F(x)=\inf_{n\in\NN} F_n(x),
$$
where\footnote{
The function $F(x)$ in another guise (see Proposition~\ref{prop:altoptfn}) appeared in the study of a certain cyclic inequality
\cite{Sadov_2022_maxcyc}, which motivated this paper.} 
\begin{align}
&F_n(x)=\inf_{t_1,\dots,t_{n-1}\geq 0} S(t_1,\dots,t_{n-1},x),
\label{op_main}
\\[1ex]
%\left(
& S(t_1,\dots,t_n)= t_1+\frac{t_2}{t_1+1}+\dots+\frac{t_{n-1}}{t_{n-2}+1}+\frac{t_n}{t_{n-1}+1}
\label{S1}
% \right).
\end{align}
In other words, $C=F(x)$ is the best constant, independent of $n$ and $\{t_j\}$, in the inequality
$$
 t_1+\frac{t_2}{t_1+1}+\dots+\frac{t_{n-1}}{t_{n-2}+1}+\frac{x}{t_{n-1}+1}\geq C.
$$ 

We will use the notation
$$
 \ab{x}=\dist(x,\ZZ).
$$
The function $x\mapsto\ab{x}$ is a $1$-periodic piecewise-linear, continuous function oscillating between $0$ and $1/2$.

\begin{theorem}
\label{thm:main}
Let
$
 u=\log x.
$
There exist numerical constants
$$
\ba{l}
 A\approx 1.7046560372, % 72-> 718
\\
 b\approx 0.6973885601,  % 01 ->0098
\ea
$$
such that 
the function $F(x)$
has the asymptotics 
\begin{equation}
\label{asf}
F(x)=eu-A+\frac{e}{2}\,\frac{{\ab{u+b}}^2}{u}
+O\left(\frac{1}{u^2}\right)
\end{equation}
as $x\to\infty$.
\end{theorem}

Theorem~\ref{thm:main} will be better understood 
in the context of two 
%simpler 
theorems stated below: 
Theorem~\ref{thm:AMGM}, which is a simpler result of the same kind, and Theorem~\ref{thm:paramcurves} of a technical nature, which
shows that a periodic (up to a small error)
remainder term appears in a class of optimization problems
concerning the lower envelope of parametric curves. 
The asymptotic formula \eqref{asf} will be eventually obtained through Theorem~\ref{thm:paramcurves}.

\medskip
Let us consider a simpler analogue of the functions
$F_n(x)$: 
\begin{equation}
\label{ep_AM-GM}
F^{(0)}_n(x)=\inf_{t_1,\dots,t_{n-1}>0} \left(
 t_1+\frac{t_2}{t_1}+\dots+\frac{t_{n-1}}{t_{n-2}}+\frac{x}{t_{n-1}}
 \right).
\end{equation}

The definition of $F^{(0)}_n(x)$ can be written as
$$
 F^{(0)}_n(x)=\inf_{a_1,\dots,a_{n}>0}
(a_1+\dots+a_n)
\quad\text{subject to $\;a_1\cdot\dots\cdot a_n=x$}.
$$

One recognizes the constrained optimization problem associated with the inequality between the arithmetic and geometric means, hence
$$
 F^{(0)}_n(x)=n x^{1/n}.
$$ 

\begin{theorem}
\label{thm:AMGM}
Let $u=\log x$. The function
$$
 F^{(0)}(x)=\inf_{n\in\NN} F^{(0)}_n(x)
$$
has the asymptotics
\begin{equation}
\label{asf0}
 F^{(0)}(x)=eu+\frac{e}{2}\,\frac{\ab{u}^2}{u}+O\left(\frac{1}{u^2}\right)
\end{equation}
as $x\to\infty$.
\end{theorem}

\begin{figure}
\begin{picture}(260,175)
\put(0,0){\includegraphics[scale=0.5]{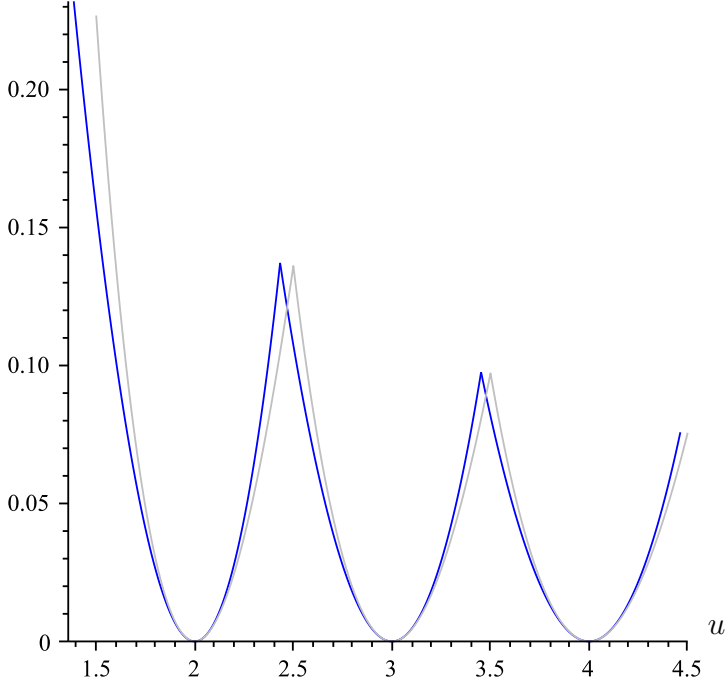}}
% Origing: file xieta.mw
\put(265,16){$u$}
\end{picture}
\caption{Illustration of Theorem~\ref{thm:AMGM}. Blue line: the function $u\mapsto F^{(0)}(x)-e u$, $u=\log x$. Gray line: %vs 
the correction term $(e/2){\ab{u}}^2/u$.}
\label{fig:f0corr}
\end{figure}

The asymptotic approximation \eqref{asf0} is illustrated in Fig.~\ref{fig:f0corr}.

%The other theorem, on surface, looks very differently, but it shows how periodic remainder terms in an asymptotics arise in some rather general situation. We will eventually reduce the problem of estimating $F_n(x)$ to this form.

\medskip
Theorems~\ref{thm:main} and \ref{thm:paramcurves} both
deal with a function defined as infimum over a family.
For instance, the graph of the function $F^{(0)}(x)$ is the lower envelope of the curves $y=nx^{1/n}$, which can be written parametrically as $\log x=nt$, $y=ne^t$. From this point of view Theorem~\ref{thm:AMGM} is an easy consequence of our next theorem. To keep things simple where possible, we will prove Theorem~\ref{thm:AMGM} directly; however, Theorem~\ref{thm:paramcurves} will be fully relevant for the proof of Theorem~\ref{thm:main}.
The relevance of the general asymptotic pattern
\eqref{asfabs} is apparent already.

\begin{theorem}
\label{thm:paramcurves}
Consider a family of parametric curves
$u=\xi_n(t)$, $v=\eta_n(t)$, $t\in I$, where $I$ is a segment of the real line. 
Suppose the functions $\xi_n(t)$, $\eta_n(t)$ have the asymptotic behaviour%
%\footnote{This form, without terms of order $n^{-1}$, will suffice in our application, where the remainder will be exponentially small. The presence of terms of order $n^{-1}$ will only change the formulas for coefficients in \eqref{asgenPhi}, see Appendix~\ref{app:paramcurves}.}
% 
$$
\ba{l}
\dst
 \xi_n(t)=p_0(t) n+ q_0(t)+ r_0(t)n^{-1}+ O(n^{-2}),
\dst
\\
 \eta_n(t)=p_1(t) n+ q_1(t)+ r_1(t)n^{-1}+ O(n^{-2})
\ea
$$
as $n\to\infty$, uniformly in $t\in I$.

Let $v=f(u)$ be the lower envelope of the curves so defined. That is, given $u$, we determine the set%
\footnote{This set is nonempty if $u$ is sufficiently large.}
of those $n$ for which the equation $\xi_n(t)=u$ has a solution and define
$$
 f(u)=\inf_{(n,t):\, \xi_n(t)=u} \eta_n(t).
$$

Denote
$$
 \beta(t)=\frac{p_1(t)}{p_0(t)}, 
\qquad 
\delta(t)=p_0(t)r_1(t)-p_1(t)r_0(t).
$$
We make the following assumptions.

\smallskip
{\rm(i)} The 
%functions $\xi_n(\cdot)$, $\eta_n(\cdot)$ and the 
coefficients $p_i(\cdot)$, $q_i(\cdot)$, $r_i(\cdot)$ are of class $C^2(I)$. 

\smallskip
{\rm(ii)}  $p_0(t)>0$, $p_1(t)>0$, $p_0(t)'>0$ on $I$.

\smallskip
{\rm(iii)} The function $\beta(t)$ has the unique point of minimum $t=t_0$ on $I$ and
$$
\beta(t)=b_0+\frac{b_2}{2}(t-t_0)^2+o(t-t_0)^2
\quad \text{near $t_0$}.  
$$

Then %the asymptotic behaviour of the function $f(u)$ is described as follows:
\beq{asfabs}
 f(u)=a_0 u+a_1+\frac{\Phi(u)}{u}+O\left(\frac{1}{u^2}\right)
\quad\text{as $u\to\infty$},
\eeq
where
\beq{coefa0a1}
 a_0=b_0,
\qquad
 a_1=q_1(t_0)-b_0 q_0(t_0),
\eeq
and
%\beq
%\ba{c}\dst
\begin{align}
& \Phi(u)=a_2+a_3\ab{\frac{u-q_0(t_0)}{p_0(t_0)}}^2,
\label{asgenPhi}
%\eeq
%and
%\beq{coefa2a3}
\\[2ex]\dst
& a_2
=-\frac{(q_1'(t_0)-b_0 q_0'(t_0))^2}{2b_2}+\delta(t_0),
\qquad
a_3=\frac{b_2}{2}\,\left(\frac{(p_0(t_0))^2}{p_0'(t_0)}\right)^2.
\label{coefa2a3}
%\eeq
\end{align}
\end{theorem}

Proof of Theorem~\ref{thm:paramcurves} is given in Appendix~\ref{app:paramcurves}.

\medskip
Our final theorem concerns an interpolation between
Theorems~\ref{thm:main} and~\ref{thm:AMGM}.

Indeed, it is natural to treat the functions $F^{(0)}(x)$
and $F(x)=F^{(1)}(x)$ as members of a one-parameter family
of functions
$$
 F^{(p)}(x)=\inf_{n\in\NN} F^{(p)}_n(x)
$$
with any $p\geq 0$,
where
$$
 F^{(p)}_n(x)=\inf_{t_1,\dots,t_{n-1}\geq 0,\;t_n=x} 
S^{(p)}_n
%\left(
(t_1,\dots,t_n),
% t_1+\frac{t_2}{t_1+p}+\dots+\frac{t_{n-1}}{t_{n-2}+p}+\frac{x}{t_{n-1}+p}
% \right).
$$
%Denote
\begin{equation}
\label{Snpt}
 S^{(p)}_n(t_1,\dots,t_n)=t_1+\sum_{j=2}^{n}\frac{t_j}{t_{j-1}+p}.
\end{equation}
In particular,
$$
 F^{(p)}_1(x)=x,\qquad F^{(p)}_2(x,p)=\inf_{t\geq 0}\left(t+\frac{x}{t+p}\right)=\begin{cases}
2\sqrt{x}-p, \quad x\geq p^2,\\
x/p, \qquad x\leq p^2.
\end{cases}
$$

We will not chase the asymptotics of $F^{(p)}(x)$ as precisely as we did in Theorems~\ref{thm:main}--\ref{thm:AMGM} but rather
focus on the constant term in the asymptotics. In particular, we determine its behaviour as $p\to +0$. We also reveal that the parameter value $p=1$ separates the regions with different analytic form of $F^{(p)}$: it is much simpler for $p>1$. This latter fact gives an additional support to our attention to the function $F(x)=F^{(1)}(x)$ as the main subject of this paper. 

\begin{theorem}
\label{thm:genpar}
Let $u=\log x$. 

\smallskip
{\rm(a)} For any $p>0$ 
\begin{equation}
\label{Fpasym}
 F^{(p)}(x)=eu-A(p)+O(u^{-1}) \quad\text{as $x\to\infty$}.
\end{equation}
The function $A(\cdot)$ is increasing.

\smallskip
{\rm(b)} If $p>1$, then
\begin{equation}
\label{Fpgt1}
 F^{(p)}(x)=F\left(\frac{x}{p}\right),
\end{equation}
so $A(p)=A+e\log p$.

\smallskip
{\rm(c)} If $0<p<1$, then
the function $A(p)$ satisfies the functional equation
\begin{equation}
\label{feq-Ap}
 A(p)=\max_{1\leq u\leq e} \left(A\left(\frac{p}{u}\right)-u+e\log u\right)+p.
\end{equation}
The asymptotics of $A(p)$ as $p\to +0$ is
$$
A(p)=\frac{p}{1-e^{-1}}+o(p).
$$
\end{theorem}

%(Note: the constant $A=A(1)$ in Theorem~\ref{thm:main} has the form $(e-1)+\eps$, where $\eps\ll 1$. Is there an explanation? Yes (partially) in Th.~\ref{thm:genpar}. Is $A(p)$ convex?

\medskip
The plan of the paper is as follows.

\smallskip
In Section~\ref{sec:variants} we describe a few guises in which the optimization problem leading to the functions $F_n(x)$ or $F^{(p)}_n(x)$ may arise. In particular,
we discuss a ``least action formulation'' and introduce a recurrence that determines the ``extremal trajectory'' for each $n$.

\smallskip
In Section~\ref{sec:simpleproofs} we prove what is easy to prove: Theorem~\ref{thm:AMGM} and the crude asymptotics
$F(x)=e\log n+O(1)$. 

\smallskip
In Section~\ref{sec:minimizers} we begin the analysis
of candidate minimizers (extremal trajectories) for
the main problem. We identify some special points
for every $n$ that play specific roles in the analysis.
We illustrate the introduced concept with graphs
and point out some pitfalls.

\smallskip
In Section~\ref{sec:proofmain} the facts observed in
Section~\ref{sec:minimizers} are fully formalized and proved and the proof of Theorem~\ref{thm:main} is completed by reducing it to a form suitable for application of Theorem~\ref{thm:paramcurves}.
 
\smallskip
In Section~\ref{sec:Ap} we prove Theorem~\ref{thm:genpar}
and describe an efficient algorithm, based on a recurrence,
to tabulate the function $A(p)$ for $0<p<1$.

\smallskip
There are two appendices containing proofs technically unrelated to the main part of the paper.

\smallskip
In Appendix~\ref{app:paramcurves} Theorem~\ref{thm:paramcurves} is proved. 

\smallskip
In Appendix~\ref{app:alpha} we rigorously prove a visually obvious property (Proposition~\ref{prop:alpha})
of a function $\alpha_\infty(x)$, which is important in the main proof.

The proof uses contour integration for estimating derivatives of a real function.
Some elements of the proof refer to numerical evaluations but always in a ``provable'' way.

\section{Alternative forms of the extremal problem}
\label{sec:variants}

\subsection{Prototype: formulations for the AM-GM}

The AM-GM optimization problem as stated in \eqref{ep_AM-GM} involves
only ``soft'' constraints $t_j>0$.
The objective function can be written more symmetrically by introducing two extra indeterminates $t_0$ and $t_n$ and subjecting them to the artificial, boundary-value constraints:  
\begin{align}
& F^{(0)}_n(x)=\min_{\mathbf{t}>0,\; t_0=1,t_n=x} \hat S_n^{(0)}(t_0,t_1,\dots,t_n),
\label{optfn0}
\\[3ex]
& \hat S_n^{(0)}(t_0,t_1,\dots,t_n)\eqbydef
\frac{t_1}{t_0}+\frac{t_2}{t_1}+\dots+\frac{t_{n-1}}{t_{n-2}}+\frac{t_n}{t_{n-1}}
\label{L0n}
.
\end{align}

A proof of the AM-GM inequality by R.~Bellmann's dynamic programming approach%
\footnote{\cite[\S~7]{BeckBel_1961} presents a dual dynamic programming formulation of the AM-GM inequality, in which one maximizes the product of $n$ indeterminates subject to the prescribed value of their sum.}
 amounts to replacing the multivariate optimization problem by
a sequence of univariate optimization problems, 
where the functions $F_n^{(0)}$ are defined recurrently by
putting $F^{(0)}_1(x)\eqbydef x$ and 
\begin{equation}
\label{recurf0}
 F^{(0)}_n(x)=\min_{y> 0}\left(F_{n-1}^{(0)}(y)+\frac{x}{y}\right),
\quad n\geq 2. %,
\end{equation}
%starting with $F^{(0)}_1(x)=x$.

\smallskip
%Similarly, there are different ways,
%besides \eqref{op_main}, to define the function
%$F_n(x)$. We %will 
%introduce them in an order that suits the convenience of proofs.

\subsection{Formulation with boundary constraints}
\label{ssec:bconst}

The constrained optimization problem that parallels
\eqref{optfn0} is
%with uniform summands in the objective function,
\begin{align}
& F_n(x)=\inf_{\mathbf{t}\geq 0,\; t_0=0,t_n=x}
\hat S_n(t_0,t_1,\dots,t_n), 
\label{optfn}
\\[3ex]
& \hat S_n(t_0,t_1,\dots,t_n)\eqbydef
\frac{t_1}{t_0+1}+\frac{t_2}{t_1+1}+\dots+\frac{t_{n-1}}{t_{n-2}+1}+\frac{t_n}{t_{n-1}+1}
\label{Ln}
.
\end{align}

In Sec.~\ref{ssec:leastaction} we will elaborate on this
presentation of the problem.

\subsection{Formulation with additive constraint}

\begin{prop}
\label{prop:altoptfn}
For every $n\in\NN$, the function $\bar F_n(x)$ defined by
\begin{equation}
\label{op_additive}
 \bar F_n(x)=\inf_{\{u_1,\dots,u_n\mid \sum u_j=1\}} \left(\frac{u_1}{u_2}+\dots+\frac{u_{n-1}}{u_n}+xu_n\right)
\end{equation}
is identical to $F_n(x)$. 
\end{prop}

The objective function in \eqref{op_additive}
can be written as $\hat S^{(0)}_{n}(1/u_1,\dots,1/u_n, x)$,
using the notation \eqref{L0n}. Thus \eqref{op_additive}
can also be seen as a result of the replacement of
the boundary condition $t_0=1$ in \eqref{optfn0} by the nonlocal constraint $1/t_0+\dots+1/t_{n-1}=1$.

%Putting $v_j=1/u_{j}$, we obtain an yet another presentation
%$$
%\bar F_n(x)=\inf_{\{v_1,\dots,v_n\mid \sum 1/v_j=1\}} \left(\frac{v_2}{v_1}+\dots+\frac{v_{n}}{v_{n-1}}+\frac{x}{v_n}\right)
%$$

The functions $F_n(x)$ in the guise \eqref{op_additive} appeared in my work \cite{Sadov_2022_maxcyc}.

\begin{proof}
\iffalse
Clearly, $\tilde f_1(x)=x$. We make the inductive assumption that
$$
 \tilde f_{n-1}(y)=\min_{\{p'_1,\dots,p'_{n-1}\mid \sum p'_j=1\}} \left(\frac{p'_1}{p'_2}+\dots+\frac{p'_{n-2}}{p'_{n-1}}+yp'_{n-1}\right).
$$
Now we put $p_n=(1+y)^{-1}$ and 
$$
 p_j=\frac{p'_j}{y+1},\quad j=1,\dots,n-1.
$$
\fi
%%%%%%%%%%%%%%%%
\iffalse
Consider the case $n=2$. Here
$$
 f_2(x)=\min_{t_1>0}\left(t_1+\frac{x}{t_1+1}\right),
$$
while
$$
 \tilde f_2(x)=\min_{p_1+p_2=1}\left(\frac{p_1}{p_2}+xp_2\right).
$$
Put
$$
 p_2=\frac{1}{t_1+1},\qquad p_1=t_1 p_2=\frac{t_1}{t_1+1}.
$$
Clearly, $p_1+p_2=1$. Conversely, given $p_1\in(0,1)$
and $p_2=1-p_1$, put
$t_1=p_1/p_2$.
\fi
%%%%%%%%%%%%%%%%
Denote the objective function in \eqref{op_additive}
%the definition of $\bar F_n(x)$ 
by $\bar S_n(\dots)$.
%Keeping $x$ fixed, 
We will exhibit a one-to-one correspondence 
between the points $(u_1,\dots,u_n)$ of the simplex $\sum u_j=1$ and the points
$(t_1,\dots,t_{n-1})$ of  $\RR_+^{n-1}$ preserving the
objective function: $S_n(t_1,\dots,t_{n-1},x)=\bar S_n(u_1,\dots,u_n,x)$. 
%\eqref{op_main}

Denote
$$
 s_j=u_1+\dots+u_j \quad (j=1,\dots,n).
$$

To a vector $\mathbf{u}$ with $\sum u_j=1$ (that is,
$s_n=1$) we
put in correspondence the vector $\mathbf{t}=(t_{1},\dots,t_{n-1})$ with
coordinates
$$
 t_{j}=\frac{s_j}{u_{j+1}}\quad (j=1,\dots,n-1).
$$
Then 
$$
 t_j+1=\frac{s_j+u_{j+1}}{u_{j+1}}=\frac{s_{j+1}}{u_{j+1}},
$$
so
$$
S_n(t_1,\dots,t_{n-1},x)=\frac{s_1}{u_2}+\sum_{j=2}^{n-1}
\frac{u_j}{u_{j-1}}+\frac{xu_n}{s_n}=
\bar S_n(u_1,\dots,u_n,x),
$$
since $s_1=u_1$ and $s_n=1$.

\iffalse
$$
\frac{t_j}{t_{j-1}+1}=\frac{p_j}{p_{j+1}}
\quad (2\leq j\leq n-1)
$$
and
$$
 t_1=\frac{s_1}{p_2}=\frac{p_1}{p_2}.
$$
Finally,
$$
 \frac{x}{t_{n-1}+1}=\frac{xp_n}{s_n}=xp_n.
$$
We see that the objective function from
the definition of $\tilde f_n(x)$  is transformed term-wise to the objective function from \eqref{op_main}.
\fi

%Conversely, given $s_n>s_{n-1}>\dots>s_1$
%we have $u_j=s_{j}-s_{j-1}$ ($j=2,\dots,n$) and $u_1=s_1$.

The inverse transformation $\mathbf{t}\mapsto \mathbf{u}$ is defined by
the formulas
$$
 u_n=\frac{1}{t_{n-1}+1},
$$
then recurrently for $j=n-1,\dots,2$
$$
u_j=\frac{t_j u_{j+1}}{t_{j-1}+1}
%=\frac{t_j t_{j+1}\dots t_{n-1}}{(t_{j-1}+1)(t_j+1)\dots(t_{n-1}+1)}.
=\frac{1}{t_{j-1}+1}\prod_{i=j}^{n-1}\frac{t_i}{t_i+1},
$$
and lastly,
$$
 u_1=t_1 u_2.
$$
We prove by induction (on ascending $j$) that 
$s_j=t_j u_{j+1}$, $1\leq j\leq n-1$.
It is true for $j=1$, since $s_1=u_1$. The induction step ($2\leq j\leq n-1)$ goes: $s_j=s_{j-1}+u_j=(t_{j-1}+1)u_j=t_j u_{j+1}$.
Finally,
$s_n=(t_{n-1}+1)u_n=1$, so that 
the $\mathbf{u}$ satisfies the required constraint.
\end{proof}

\subsection{Existence of a minimizer}
\label{ssec:existence}
We have written `$\min$' in Eq.~\eqref{optfn0}, since the
minimum of the objective function in is attained at critical point where $t_j=x^{j/n}$ ($j=0,1,\dots,n)$. It is also true, thouhg not immediately obvious, that the greatest lower bound of the objective function in \eqref{optfn} is attained. We prove this now along with first elementary properties of the functions $F_n(\cdot)$.

%%%%%%%%%%%%%%%%%%%%%%%%
\iffalse
We can, of course, introduce an interpolatory family of optimization problems:
$$
\ba{l}\dst
 f_n(x|s)=
\min_{t_1,\dots,t_{n-1}\geq 0} \left(t_1+\frac{t_2}{t_1+s}+\dots+\frac{t_{n-1}}{t_{n-2}+s}+\frac{x}{t_{n-1}+s}
\right),
\\[2ex]
f_0(x|s)=x.
\ea
$$
Equivalently (upon substitution $t_j\mapsto st_j$): $f_n(x|s)=g_n(x/s\mid s)$, where
$$
\ba{l}\dst
 g_n(x|s)=
\min_{t_1,\dots,t_{n-1}\geq 0} \left(st_1+\frac{t_2}{t_1+1}+\dots+\frac{t_{n-1}}{t_{n-2}+1}+\frac{x}{t_{n-1}+1}
\right),
\\[2ex]
g_1(x|s)=sx.
\ea
$$
There is the recursive definition, the same as one for $f_n$:
$$
 g_n(x|s)=\min_{y>0}\left(g_{n-1}(y)+\frac{x}{y+1}\right),\quad n\geq 2.
$$
Only the initial condition (case $n=1$) depends on $s$.
\fi
%%%%%%%%%%%%%%%%%%%%%%%%

\begin{prop}
\label{prop:monotfn}
{\rm(a)} For every $n$, the function $F_n(\cdot)$ is nondecreasing.

{\rm(b)} For every fixed $x$, the sequence $(F_n(x))$ is nonincreasing, hence the function $F(x)$ can be defined as
a monotone limit
$$
F(x)=\downarrow\lim_{n\to\infty} F_n(x).
$$

{\rm(c)} 
\iffalse
For every $n$ the inequality 
\begin{equation}
\label{compf0fn}
\frac{1}{2}F^{(0)}(x)\leq F_n(x)\leq F_n^{(0)}(x)
\end{equation}
holds. Consequently,
\begin{equation}
\label{compf0f}
\frac{1}{2}F^{(0)}(x)\leq F(x)\leq F^{(0)}(x).
\end{equation}

%
%$f_n(x)=x$ if $x\leq 1$ for any $n$.

{\rm(d)} 
\fi
The objective function $(t_1,\dots,t_{n-1})\mapsto L_n(0,t_1,\dots,t_{n-1},x)$ attains its minimum value at some nonnegative
$(n-1)$-tuple.
Hence the symbol `$\inf$' in  \eqref{op_main}, \eqref{optfn} and \eqref{op_additive} can be replaced by `$\min$'. 
\end{prop}

\begin{proof}
(a) Trivial: the function $x\mapsto S_n(t_0,\dots,t_{n-1},x)$ is increasing for every
$n$-tuple 
$(t_0,\dots,t_{n-1})$.

\smallskip
(b) The inequality $F_{n}(x)\leq F_{n-1}(x)$ 
is due to the fact that imposing the additional constraint $t_1=0$ in \eqref{optfn} yields $F_{n-1}(x)$. 
That is,
$$
L_{n-1}(0,t_2,\dots,t_n)=L_n(0,0,t_2,\dots,t_n).
$$

\iffalse
(c) 
The right inequality in \eqref{compf0fn} is obvious: $L_n^{(0)}(\mathbf{t})>L_n(\mathbf{t})$ for any
$\mathbf{t}>0$.

The left inequality is equivalent to the claim:
for any $n$, $x$ and $t_1,\dots,t_{n-1}$ there exists $k$ such that
$$
 L_n(0,t_1,\dots,t_{n-1},x)\geq \frac{1}{2} F_k^{(0)}(x).
$$
We prove this by induction on $n$.
The base case $n=1$ is obvious, since 
$L_1(0,x)=x$ and one can take $k=1$.

Assuming that the claim is true with $n-1$ in place of $n$ and some $k'$ in place of $k$, we write
$$
 L_n(0,t_1,\dots,t_{n-1},x)=
L_{n-1}(0,t_1,\dots,t_{n-1})+\frac{x}{t_{n-1}+1}.
$$
If $t_{n-1}\leq 1$, then $x/(t_{n-1}+1)\geq x/2$,
so the claim is true with $k=1$.
%
If $t_{n-1}> 1$, then $x/(t_{n-1}+1)> x/(2t_{n-1})$.
%
Using the recursive definition \eqref{recurf0}, we infer
$$
%\ba{l}\dst
 S_n(0,t_1,\dots,t_{n-1},x)
%=S_{n-1}(0,t_1,\dots,t_{n-1})+\frac{x}{t_{n-1}+1}
%\\[3ex]\dst
%\qquad
\geq
%\frac{1}{2} f^{(0)}(x)\frac{x}{2t_{n-1}}=
\frac{1}{2}\left(f_{k'}^{(0)}(x)+\frac{x}{t_{n-1}}\right)
\geq \frac{1}{2}f_{k}^{(0)}(x),
%\ea
$$
where $k=k'+1$. 

The left-hand side of \eqref{compf0fn} does not depend on $n$, hence  \eqref{compf0fn} implies \eqref{compf0f}.

\smallskip
(d) 
\fi

(c) 
It suffices to show that the optimization region
in \eqref{op_main} can be reduced to a compact.
Let us fix $x$ and put $C=F_n(x)$. 
If $t_1>C+1$, then $S(t_1,\dots,t_{n-1},x)>C+1$,
so the optimization region can be reduced to
$\{0\leq t_1\leq C+1,\;t_2,\dots,t_{n-1}\geq 0\}$.

Suppose that $t_1\leq C_1=C+1$ and define the constants
$C_k$ recurrently: $C_k=(C_{k-1}+1)(C+1)$.
If $k\in\{2,\dots,n-1\}$ is the least index such that $t_k>C_k$ (assuming such an index exists),
we have $t_k/(t_{k-1}+1)>C+1$. Hence $S_n(t_1,\dots,t_{n-1},x)>C+1$ in the region 
$\{t_1>C_1\}\cup\{t_2>C_2\}\cup\dots\{t_{n-1}>C_{n-1}\}$.
Therefore the optimization region is reduced
to the parallelotop $\{0\leq t_k\leq C_k,\;\;k=1,\dots,n-1\}$.
\end{proof}

\begin{definition}
An $(n-1)$-tuple $(t_1,\dots,t_{n-1})$ is called a {\em minimizer}\ for the problem \eqref{op_main} if
$S_{n-1}(t_1,\dots,t_{n-1},x)=F_{n-1}(x)$.
We will also say ``a minimizer for $F_n(x)$''. 

With reference to the problem \eqref{optfn} we will call
the minimizing $(n+1)$-tuple $(0,t_1,\dots,t_{n-1},x)$
a minimizer. 
\end{definition}

%Minimizers will play the decisive role in the proof of Theorem~\ref{thm:main}. 
In preparation to the proof of Proposition~\ref{prop:funeqf} let us prove the first result
on minimizers. 

\begin{lemma}
\label{lem:minimizer0}
If $x\leq 1$, then the unique minimizer for $F_n(x)$
is the zero tuple.
\end{lemma}

\begin{proof}
We need to prove that if not all $t_j$ equal zero, then
$S_{n}(t_1,\dots,t_{n-1},x)>S_{n}(0,\dots,0,x)=x$.
Equivalently, the inequality to prove is: if $\mathbf{t}\neq\mathbf{0}$, then (in the worst case, when $x=1$) 
$$
 S_{n-1}(t_1,\dots,t_{n-1})>\frac{t_{n-1}}{t_{n-1}+1}.
$$

It is obvious, if $n-1=1$. By induction we proceed from the estimate for $S_{n-1}$ to the one for $S_n$.

If $t_1=\dots=t_{n-1}=0$ and $t_{n}>0$, then
we are in the same situation as in the case $n-1=1$.
If not all $t_j$ with $j\leq n-1$ are equal to 0, then
by the induction hypothesis we have
$$
 S_{n}(t_1,\dots,t_{n})=S_{n-1}(t_1,\dots,t_{n-1})+\frac{t_n}{t_{n-1}+1}
>\frac{t_{n-1}+t_n}{t_{n-1}+1}.
$$
Since
$$
 \frac{t_{n-1}+t_n}{t_{n-1}+1}-\frac{t_{n}}{t_{n}+1}
=\frac{t_n^2+t_{n-1}}{(t_{n-1}+1)(t_{n}+1)}\geq 0,
$$
the induction step is complete.
\end{proof}

\subsection{Dynamic programming formulation}
\label{ssec:dynprog}
\smallskip
We give recurrent equations for
$F_n(\cdot)$ analogous to \eqref{recurf0}.

Also it appears possible to define the function $F(x)$ is expressed without explicit reference to the $F_n$'s
--- by means of
the functional equation \eqref{fe1}.
%Later, in Proposition~\ref{prop:fnstab}, we show that the minimization range in Eqs.~\eqref{recurf} and \eqref{fe1} 
%can be reduced to an interval of length approximately $x/e$.

\begin{prop}
\label{prop:funeqf}
Let $\tilde F_n(x)$, $n=1,2,\dots$, be a sequence of functions %such that $\tilde F_n(x)$ is 
defined for $x>0$ as follows.

\smallskip
{\rm (i)} For $0<x\leq 1$ and all $n\;\;$ $\tilde F_n(x)= x$.

\smallskip
{\rm (ii)}  
For $x>1$ and $n\geq 2$
\begin{equation} 
\label{recurf}
 \tilde F_n(x)= \min_{0< y<x-1}\left(\tilde F_{n-1}(y)+\frac{x}{y+1}\right).
\end{equation} 

Define also the function $\tilde F(x)$, $x>0$, as follows.

\smallskip
{\rm (i$'$)} If $0<x\leq 1$, then $\tilde F(x)= x$.

\smallskip
{\rm (ii$'$)} In the intervals $(1,2], (1,3], \dots$ the values $\tilde F(x)$ are defined recursively by means of the functional equation
\begin{equation} 
\label{fe1}
 \tilde F(x)=\min_{0< y<x-1}\left(\tilde F(y)+\frac{x}{y+1}\right).
\end{equation} 
(The right-hand side refers to values of $\tilde F(\cdot)$ defined earlier.)

Then \\
{\rm (a)} 
$\tilde F(x)=\tilde F_n(x)$ for $0<x\leq n$;
\\
{\rm (b)} 
$\tilde F_n(x)=F_n(x)$ for all $n$ and $x$, and
\\
{\rm (c)} 
$\tilde F(x)=F(x)$ for all $x$. 
\end{prop}

\begin{proof}
(a) $\tilde F_1(x)=\tilde F(x)=1$ for $0<x\leq 1$ by definition.
The claimed identity for any $n\geq 2$ follows by induction, comparing the recursive definitions of $\tilde F_n(x)$ and $\tilde F(x)$.

\smallskip
(b)
We will prove by induction, simultaneously, that (I) $\tilde F_n(x)=F_n(x)$,
$0<x\leq n$; and (II) $0<t_{n-1}^*<x-1$, so that $t_{n-1}^*$ can be identified with the point of minimum
$y_*$ in the right-hand side of \eqref{recurf}.

For $n=1$, the claim (I) is trivial and the claim (II) is vacuous. %We may assume that $x>1$.

Suppose that the objective function in \eqref{op_main} attains its minimum at $\mathbf{t}^*=(t_1^*,\dots,t_n^*)$,
so that
$$
 F_n(x)=S_{n-1}(t^*_1,\dots,t^*_{n-1})+\frac{x}{t^*_{n-1}+1}.
$$
The existence of such $\mathbf{t}^*$ is ensured by Proposition~\ref{prop:monotfn}.

We have
$$
\frac{\partial S_n(\mathbf{t},x)}{\partial t_{n-1}}
=\frac{\partial}{\partial t_{n-1}}
\left(\frac{t_{n-1}}{t_{n-2}+1}+\frac{x}{t_{n-1}+1}\right)=\frac{1}{t_{n-2}+1}-\frac{x}{(t_{n-1}+1)^2}.
$$
(For $n=2$, we formaly take $t^*_{0}=0$.) 

Since $x>1$, there can be no minimum at $t_{n-1}=0$. 
Hence $t^*_{n-1}>0$. 

By the choice of $\mathbf{t}^*$ we have $S_{n-1}(t^*_1,\dots,t^*_{n-2},t^*_{n-1})
=F_{n-1}(t_{n-1}^*)$. In order to complete the induction step
%, i.e.\ to obtain Eq.~\ref{recurf} 
it remains to prove that $t^*_{n-1}<x-1$.

The necessary condition of extremum yields
$$
 (t^*_{n-1}+1)^2=x(t^*_{n-2}+1).
$$
Suppose that $t^*_{n-1}+1\geq x$. Then $t^*_{n-1}\leq t^*_{n-2}$. If $n=2$, we get $1<x\leq t^*_1+1\leq t^*_0+1=1$, a contradiction. 

If $n>2$ and $t^*_{n-1}>1$, then by part (II) of the induction hypothesis the last component $t^*_{n-2}$ in the minimizer for $F_{n-1}(t^*_{n-1})$ satisfies the inequality $t^*_{n-2}<t^*_{n-1}-1$.

Finally, if $n>2$ and $t^*_{n-1}\leq 1$, then by Lemma~\ref{lem:minimizer0} $t^*_{n-2}=0$.

In both cases, we arrive to a contradiction with assumption
$t^*_{n-1}\leq t^*_{n-2}$. 

The induction step is complete.

\smallskip
(c) By parts (a) and (b) we get $F_n(x)=\tilde F(x)$
for any integer $n\geq x$. Since the sequence $F_n(x)$
is monotone (Proposition~\ref{prop:monotfn}(b)), we conclude that $F(x)=\tilde F(x)$.
\end{proof}

\subsection%
{General look: least action formulation}
\label{ssec:leastaction}

\iffalse
%%%%%%%%%%%%%%%%%%%%%%
Consider the optimization problem in the form~\eqref{op_main}. A {\em minimizer} is .
The existence of a minimizer 
has been proved in Proposition~\ref{prop:altoptfn}.
Since there are no constraints in the form of equalities here, there are two possibilities for a minimizer:
(i) a minimizer %$(t_1,\dots,t_{n-1})$ 
lies in the interior of the admissible domain;
% , that is all $t_i>0$; 
then it is a critical point of the 
% where all partial derivatives of the 
objective function;
% turn to $0$, 
(ii) a minimizer is a boundary point, that is, at least one of $t_j$ ($1\leq j\leq n-1$) is zero. 

Case (ii) yields a minimization problem with fewer variables and its critical points are to be analysed.

For a particular value of $n$ there may be more than one candidate point for a minimizer, and subsequent minimization over $n$ (to find a minimizer for $f(x)$)
potentially involves another act of selection. 
\fi
%%%%%%%%%%%%%%%%%%%%%

%Let us temporarily 
The boundary value formulation in Sec.~\ref{ssec:bconst} and the dynamic programming formulation in Sec.~\ref{ssec:dynprog} do not by themselved offer
much of a progress in finding the asymptotics of $F(x)$.
We will need to study minimizers. For that purpose it
is useful to look at the optimization problem from a more general point of view.

Consider a minimization problem with objective function of the form
\beq{Sn-general}
\mathcal{S}_n(\mathbf{t})=\sum_{j=1}^n L(t_{j-1},t_j).
\eeq 
By analogy with classical mechanics, we call the function $L(\cdot,\cdot)$ the {\em Lagrangian}.  
A {\em virtual trajectory}\ $\mathbf{t}=(t_0,\dots,t_n)$
is subject to the constraints $t_0=x_0$, $t_n=x$;
the boundary values $x_0$ and $x$ are assumed given.
The function $\mathcal{S}_n(\mathbf{t})$ (``integral of the Lagrangian along a virtual trajectory'') is the analog of {\em action} in mechanics. The object of our attention is the extremal value
\beq{fn-general}
 g_n(x_0,x)=\min_{(t_0,\dots,t_n)\mid\, t_0=x_0,\,t_n=x}
\mathcal{S}_n(t_0,\dots,t_n).
\eeq
We recognize in this setting ``the least action principle'' with discrete time.

In mechanics, one usually pays little attention to the minimum value of the action as such; the goal is to determine the extremal trajectory, which describes the actual evolution of a mechanical system. 
Here, we are originally interested in extremal values $g_n$, but the focus will eventually shift to extremal trajectories.

We assume that the Lagrangian is differentiable in its domain and denote $L_1(u,v)=\partial L(u,v)/\partial u$
and $L_2(u,v)=\partial L(u,v)/\partial v$.

Let $\mathbf{t}^*$ be an extremal (more precisely, minimizing) trajectory
for the problem \eqref{fn-general}. Suppose that $(t^*_1,\dots,t^*_{n-1})$ is an interior point of the domain of the function $\mathcal{S}_n|_{t_0=x_0,\;t_n=x}$.
The necessary conditions of extremum $\partial \mathcal{S}_n/\partial t_j|_{\mathbf t=\mathbf t^*}=0$ ($1\leq j\leq n-1$), known in mechanics as the Euler-Lagrange equations, in the expanded form read
\beq{ELeq-general}
 L_2(t_{j-1},t_j)+
 L_1(t_{j},t_{j+1})=0,\quad j=1,\dots,n-1.
\eeq

Suppose that the equation $L_1(u,v)+p=0$ with given $u$ and $p$ is uniquely solvable for $v$ in all occurrences and denote the solution as $v=V(u,p)$.
Then the system \eqref{ELeq-general} can be written in the form of the second order recurrence equations
\beq{ELeq-recur}
 t_{j+1}=V(t_j,L_2(t_{j-1},t_j)),\quad j=1,\dots,n-1.
\eeq
More precisely, we have the boundary value problem comprising the equations \eqref{ELeq-recur} and the
boundary conditions
$$
 t_0=x_0,\qquad t_n=x.
$$
Introduce a free parameter $\tau$,
set $T_0(\tau)=x_0$, $T_1(\tau)=\tau$, and define  further functions $T_j(\tau)$ by 
%the recurrence, which parallels \eqref{ELeq-recur}:
making the substitutions 
$t_i\mapsto T_i(\tau)$ in the recurrence \eqref{ELeq-recur}:
\beq{defTj}
T_{j+1}(\tau)=V(T_j(\tau),L_2(T_{j-1}(\tau),T_j(\tau))), \quad j=1,\dots,n-1.
\eeq
The boundary value problem can be in principle solved by the shooting method: 
%solving 
the equation 
%(easy to say!)
%\beq{Tnx}
$
T_n(\tau)=x,
$
%\eeq 
%we 
determines the value of $\tau$ and hence the whole extremal trajectory. The existence and uniqueness of solution are two immediate concerns. We will attend to them in our concrete case --- first, in Section~\ref{sec:minimizers}, by presenting numerical results and revealing fine points, and then, analytically.
%leaving interpretation of this step for later.

The dynamic programming approach to the minimization
problem \eqref{fn-general} leads to the recurrence
\beq{recg}
%$$
 g_{j+1}(x)=\min_y (g_{j}(y)+L(y,x)),\quad j\geq 1,
%$$
\eeq
with initial condition  
%\beq{g1}
$$
 g_1(x)=L(x_0,x).
$$
%\eeq
An extremal trajectory defined by the recurrence \eqref{defTj} induces the recurrence
\beq{Gn-rec}
 G_{j+1}(\tau)=G_j(\tau)+L(T_j(\tau),T_{j+1}(\tau)),
\quad j\geq 0,
\eeq
with initial condition  
%\beq{G0}
$$
 G_0(\tau)=0.
$$
%\eeq
In a simple scenario, we expect that $G_n(\tau)$ is the minimum value $g_n(x_0,x)$ sought in \eqref{fn-general}. However, if there are several 
solutions of the equation $T_n(\tau)=x$ and
several extremal trajectories, one needs to pick the minimum among the several corresponding %different 
values $G_n(\tau)$.  Also one should not ignore the possibility that the minimum value may be attained at the boundary of the parameter domain.

\begin{remark}
%We note in passing that, 
Upon the change of variables $(u,v)\mapsto (u,p)$, $p=L_2(u,v)$ applied to all pairs $(u,v)=(t_{j},t_{j+1})$, $j=0,\dots,n-1$,
 the recurrence \eqref{ELeq-recur} can be recast in the form 
\beq{EL-H}
 p_{j+1}=L_2(t_j,V(t_j,p_j)),
\qquad t_{j+1}=V(t_j,p_j),
\qquad j=1,\dots,n-1.
\eeq
This can be viewed as a ``Hamiltonian system with discrete time'' referring to the fact that
the map $(t_j,p_j)\mapsto(t_{j+1},p_{j+1})$ is symplectic, that is, its Jacobian is equal to $1$. 
Cf.~\cite[end of Sec.~9.1]{McDuff-Salomon_1998}.
%\end{remark}

In the present paper, the Lagrangian is
$$
 L(u,v)=\frac{v}{u+1}.
$$

%\begin{remark}
Incidentally, another asymptotic problem solved recently by the author \cite{Sadov_2021} involves the Lagrangian 
$$
 L(u,v)=u+\frac{1+v}{u}
$$
and the boundary conditions $p_0=t_n=0$ for the system in the Hamiltonian form \eqref{EL-H}.%
%are fixed
%\footnote{One has to reverse the order or components of the 
%vector $\mathbf{t}$ in \cite{Sadov_2021} to make the
%comparison exact.}
\footnote{The comparison applies 
to the trajectory $(t_0,\dots,t_n)$ 
%$\mathbf{t}$ 
in \cite{Sadov_2021} %is 
re-indexed backwards.}
%in the reverse order.}
(Unlike in the present case, there is no variable $x$; the integer $n$ is the only parameter.)
Crusial to the asymptotic analysis in \cite{Sadov_2021} is the presence of a fixed point  
of the map $(u_j,p_j)\mapsto(u_{j+1},p_{j+1})$, which is not the case here. 
%
%In both cases the Lagrangian $L(u,v)$ is linear in $u$, Eq.~\eqref{ELj} is uniquely solvable for $t_{j+1}$; however, 
%technical parallels do not go much further. 
\end{remark}

%--------------------------------------------------
\section{Asymptotics that allow for simple proofs}
\label{sec:simpleproofs}

Here we prove two results that do not rely on the analysis of extremal trajectories. The material of this section is not used in the sequel.

%--------------------------------------------------
\subsection{Proof of Theorem~\ref{thm:AMGM}}
%Asymptotics of the function $f^{(0)}(x)$
\label{ssec:asf0}

%\begin{proof} 
We write $u=\log x$, as in the formulation of the Theorem.

Let $u_n=n(n+1)\log(1+n^{-1})=n+1/2+O(1/n)$. We have:
$F^{(0)}_{n+1}(x)>F^{(0)}_n(x)$ if and only if $u>u_n$. 

Suppose $n>1$, $u\in I_n=[u_{n-1},u_n]$.
Then $F^{(0)}(x)=F^{(0)}_{n}(x)=u\varphi(u/n)$,
where $\varphi(t)=t^{-1}e^t$. The function $\varphi(t)$ has the minimum at $t=1$
and $\varphi(1+\eps)=e(1+\eps^2/2)+O(\eps^3)$ as $\eps\to 0$. 

Put $u=n+s$, $|s|\leq 1/2+O(1/n)$. We have $\varphi(u/n)=e+(e/2)(s/n)^2+O(n^{-3})$. 

Note that $|s|=|\left\langle  u\right\rangle|+O(1/n)$;
indeed, $s\neq \left\langle  u\right\rangle$ can happen
only near half-integer values of $n$, where $s$ is close to $\pm 1/2$.

Since $n^{-1}=u^{-1}+O(u^{-2})$, we 
obtain the asymptotics \eqref{asf0}. 
\qed
%\end{proof}

\begin{remark}
The left and right derivative numbers of $F^{(0)}(x)$
at the endpoints of the intervals $I_n$ are different.
In Section~\ref{ssec:experiment} we will see that the situation with function $F(x)$ is more interesting in this respect.
\end{remark}

%%%%%%%%%
\iffalse
We have $f^{(0)}(x)=\log x\cdot \min_{\ZZ\ni n\geq 1}\varphi(\log x/n)$,
where $\varphi(t)=t^{-1}e^t$. 
The function $\varphi(t)$ has the minimum at $t=1$
and $\varphi(1)=e$. 

Put $u=\log x$.
The value $n_*=n_*(x)$ for which
$f^{(0)}(x)=u\varphi(u/n_*)$ is always one of the two integers nearest to $u$: either $n_*=\lfloor u\rfloor$ or $n_*=\lceil u\rceil$. Hence
$u/n_*=1+\epsilon$, where $\epsilon=O(1/u)$ but $\epsilon\neq o(1/u)$.

Since $\varphi(1+\epsilon)=e+c\epsilon^2+ o(\epsilon^2)$ with $c>0$, we conclude that $\varphi(u/n_*)=e+O(1/u^2)$ and the 
remainder estimate is optimal. The claimed asymptotics 
of $f^{(0)}(x)$ follows.
\fi
%%%%%%%%%%%%%%%%%

\subsection{Crude asymptotics of the function \texorpdfstring{$F(x)$}{F(x)}}
\label{ssec:crudeas}

\begin{prop}
\label{prop:crudeas}
%%%%%%%%%
\iffalse
Let $f(x)$ be a real-valued function defined for $x> 0$ recursively as follows:
$$
 f(x)=x, \qquad 0< x\leq 1,
$$
and
%$$
\begin{equation}
\label{recg}
 f(x)=\inf_{0<y\leq x-1}\left(f(y)+\frac{x}{y+1}\right),
 \qquad x>1.
\end{equation}
\fi
%%%%%%%%%%
The function $F(x)$ has the asymptotic behavior
$F(x)=e\log x+O(1)$ as $x\to\infty$. Specifically,
if $x\geq 1$, then
$F(x)$ satisfies the inequalities
%$$
\begin{equation}
\label{rbndf}
-a_1+\frac{b_1}{x+1}\leq
F(x)-e\log(x+1)\leq -a_2+\frac{b_2}{x+1},
\end{equation}
%$$
where 
$b_1\approx 1.77$
%1.76928$ 
is the (smaller of the two) root of the equation 
%$$
\begin{equation}
\label{eqb1}
2\log\frac{b+1}{2}=\frac{b}{e};
% Larger root: 8.431759467
\end{equation}
the constant $a_1$
is defined by
$$
a_1=\max_{0\leq x\leq 1}\left(e\log(x+1)+\frac{b_1}{x+1}-x\right)
\approx 1.78, %1.78355,  
%inf=1.767427630, sup=1.783555720
$$
and $a_2=b_2=e/(e-1)\approx 1.58$. %1.58198$.
\end{prop}

\begin{proof}
We prove that \eqref{rbndf} is true for $0<x\leq n$ by induction on $n$.  For $n=1$ the left inequality is fullfilled by the definition of $a_1$, while the left inequality is due to the fact that 
the function $y\mapsto  e\log(x+1)+b_2/(x+1)-x$ attains it minimum on $[0,1]$ at $x=0$.
%attained at $y=0$ and equals $b_2$.

For the induction step consider the function
$t\mapsto e\log t+u/t$ with given $u>0$. The critical
point is $t^*=u/e$, hence
\beq{minf_tmp}
 \min_{t>0}\left(e\log t+\frac{u}{t}\right)=
e\log t^*+\frac{u}{t^*}=
e\log u.
\eeq

Let us first prove the left inequality
in \eqref{rbndf} for $x\leq n+1$.
Using the inductive assumption and the formula \eqref{fe1} from Proposition~\ref{prop:funeqf}, we get 
$$
 F(x)\geq e\log(y+1)-a_1+\frac{b_1}{y+1}+\frac{x}{y+1}
$$
for any $y\in(0,x-1]$. By \eqref{minf_tmp} with $u=b_1+x$, we have
$$
 F(x)\geq e \log(b_1+x)-a_1.
$$
%To complete the induction step 
It remains to check the inequality
$$
e\log(b_1+x)\geq e\log(x+1)+\frac{b_1}{x+1};
$$
equivalently,
$$
\frac{\log(1+(b_1-1)s)}{s}\geq \frac{b_1}{ e},
$$
where
$s=(x+1)^{-1}$. 

The left-hand side is a decreasing function of $s$.  We may assume that $x>1$, hence $s<1/2$.
Therefore
$$
\frac{\log(1+(b_1-1)s)}{s}> 2\log(1+(b_1-1)/2).
$$
By  definition of $b_1$, the right-hand side equals $b_1/e$. The proof of the lower estimate for $F(x)$ is complete.

We prove the right inequality
in \eqref{rbndf} for $x\leq n+1$ similarly.
Again, using the inductive assumption and the formula \eqref{fe1}, we get 
$$
 F(x)\leq \min_{y<x-1} \left(e\log(y+1)-a_2+\frac{b_2}{y+1}+\frac{x}{y+1}\right).
$$
Put $u=b_2+x$ in \eqref{minf_tmp}. The critical point
$t^*=(b_2+x)/e$ corresponds to $y^*=t^*+1$. 
Now, $y^*$ lies in the admissible interval 
%of minimization 
$0<y<x-1$ provided that
$t_*< x$, that is, $x> b_2/(e-1)\approx 0.92$.
This condition is fulfilled, since we assume $x\geq 1$.
So
$$
 F(x)\leq e\log(b_2+x)-a_2.
$$
It remains to prove that
$$
e\log(b_2+x)\leq e\log(x+1)+\frac{b_2}{x+1}.
$$
Subtracting $e\log(x+1)$ from both sides, we estimate:
$$
 e\log\frac{x+b_2}{x+1}=e\log\left(1+\frac{b_2-1}{x+1}\right)< e\frac{b_2-1}{x+1}.
$$
Since $e(b_2-1)=b_2$, the proof is complete.
\end{proof}

\begin{remark}
\label{rem:ill-mainasym}
 The double-sided estimate
\eqref{rbndf} agrees with %the 
asymptotic formula \eqref{asf}, since $a_2<A<a_1$, while the rational
terms $\mathrm{const}\cdot(x+1)^{-1}$, as well as the difference between $\log x$ and $\log(x+1)$, are absorbed by the remainder. %in \eqref{asf1}.
A freedom to choose a particular form of asymptotically negligible terms can be used to 
devise an approximation to $F(x)$ that would be practical not only for large values of $x$ but also for small ones.

The approximation with an empirically chosen 
correction $b_2/(x+1)$
%$$
% f_{\rm approx}(x)= e\ln\left(x+b_2\right)-A+\frac{0.446}{(x+1)^2},
%$$
%
\begin{equation}
\label{empir_approx}
F(x)\approx
F_{\rm approx}(x)\eqbydef e\ln(x+1)-A+\frac{b_2}{x+1},
\end{equation}
is illustrated in Figure~\ref{fig:f1corr}(a).
%(The effect of the term $b_2/(x+1)$ can be appreciated by comparison with Fig.~\ref{fig:xieta1-6} from \S~\ref{ssec:experiment}.)
Note that if we had plotted only the initial part of the graph, up to approximately $\log(x+1)=8$ (shown in red), then it could
lead to a misleading assumption that already the second term of the asymptotics, that of order $O(1)$, is periodic.
The right part of the graph shows that the error of approximation in fact decays at a rate $O((\log x)^{-1})$,
as the theory predicts. The equation of the dashed majorizing curve is $y=(e/8)(\log(x+1)+7)^{-1}$. (Here $e/8=e/2\cdot\max\ab{\cdot}^2$. The ``magic'' number $7$ makes the enveloping curve closer to the graph without affecting the asymptotics up to terms of order $O((\log x)^{-1})$.)

The right part of the same graph is shown in larger scale on Fig.~\ref{fig:f1corr}(b) in comparison with third term of asymptotics from Theorem~\ref{thm:main}.
\end{remark}

%----------------------
\iffalse
\begin{figure}
\begin{picture}(260,190)
\put(0,0){\includegraphics[scale=0.5]{%empirical_approximation_error.eps}
empirical_approximation_error-logx}
}
% Origin: file xieta.mw
\put(315,25){\small $\log (x+1)$}
\end{picture}
%
\caption{Difference $f(x)-f_{\rm approx}(x)$ vs $\log(x+1)$} 
%for $0<x<37$}
\label{fig:emperror}
\end{figure}
\fi
%----------------------

\begin{figure}
\begin{picture}(311,520)
%\put(0,-5)%
\put(50,260)%
{\includegraphics[scale=0.5]{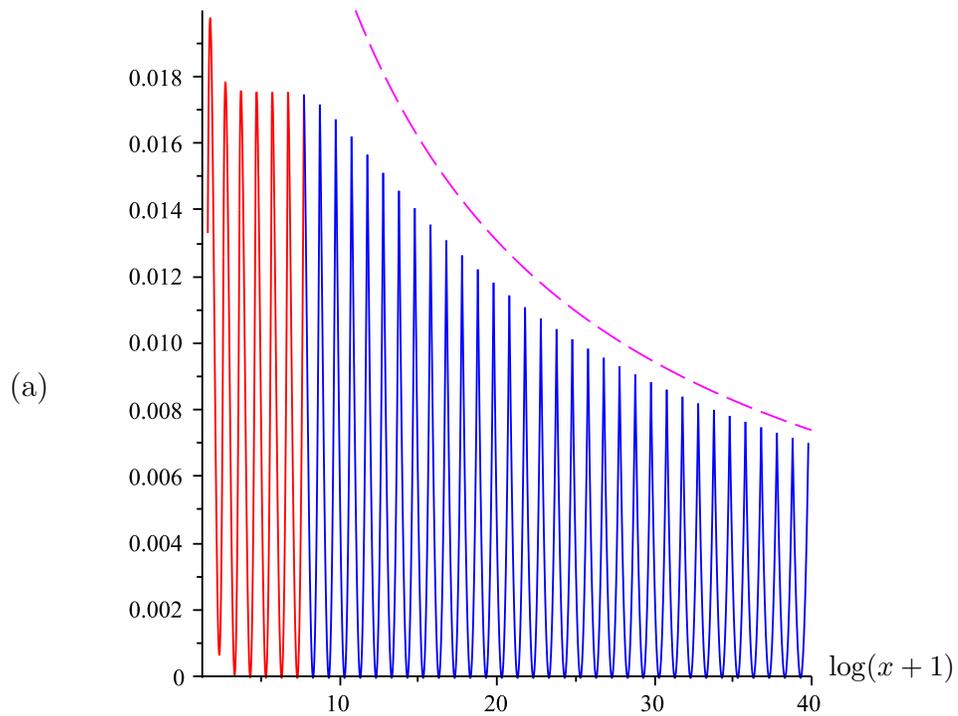}}
% Origing: file xieta.mw
%\put(265,14)%
\put(315,274)%
{\small $\log (x+1)$}
%\end{picture}
%
%\caption{The difference $F(x)-F_{\rm approx}(x)$ vs $\log(x+1)$}
%\label{fig:emperror1}
%\end{figure}

%\begin{figure}
%\begin{picture}(263,260)
\put(50,0){\includegraphics[trim={0 0 0 50pt},clip,scale=0.5]{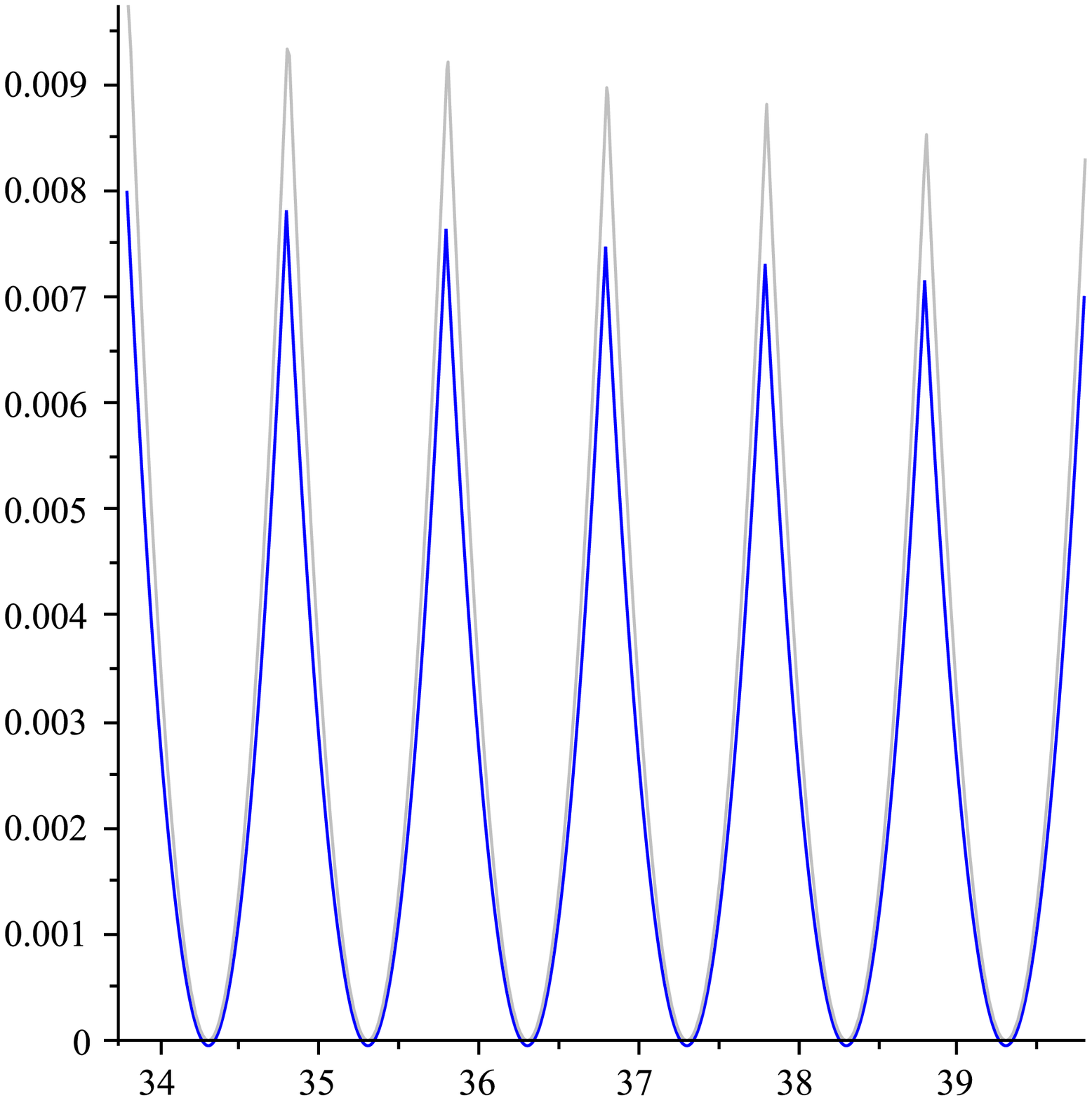}}
% Origing: file xieta.mw
\put(315,16){$\log(x+1)$}
%-------------
\put(5,380){(a)}
\put(5,120){(b)}
\end{picture}
\caption{%To Proposition~\ref{prop:crudeas} and Theorem~\ref{thm:main}. 
(a) The error of empirical approximation \eqref{empir_approx}. (b) Blue line: the same error; gray line: the correction term $\Phi(u)/u$ from Theorem~\ref{thm:main}.}
%Remark \ref{rem:ill-mainasym}
\label{fig:f1corr}
\end{figure}

%------------------------------------------------

\section{Extremal trajectories and minimizers} %; heuristics}
\label{sec:minimizers}

Here we make initial steps towards the reduction of Theorem~\ref{thm:main} to Theorem~\ref{thm:paramcurves}.

%A {\em minimizer}\ is a point where the objective function attains its minimum. 
Taking Eq.~\eqref{optfn} as the definition of $F_n(x)$, the minimizer candidates are critical points of $\hat S_n(\mathbf{t})|_{t_0=0,t_n=x}$ inside the admissible domain $\{t_i>0,\,i=1,\dots,n-1\}$ and on its boundary, where $t_i=0$ for at least one $i\geq 1$.

%The scheme of calculations is expounded in greater generality in \S~\ref{ssec:leastaction} to emphasize 
%relations not depending on the specifics of the problem at hand.

%To evaluate $f_n(x)$, we want to identify candidate points  among which a minimizer (a point
%where the objective function attains its minimum) is to be selected: these are critical points inside the admissible domain and on its boundary. 

On the other hand, from Proposition~\ref{prop:funeqf}(c) we know that $F(x)=F_n(x)$ for all sufficiently large $n$
(e.g. $n\geq x$), and we will observe the stabilization of
minimizers in a precise componentwise sense. 

In this section we will use the termiology of the general least action problem with discrete time as described in Section~\ref{ssec:leastaction}.

\subsection{Extremal trajectories: basic properties}
\label{ssec:basicrelations}

The Euler-Lagrange equations \eqref{ELeq-recur}
for the concrete problem \eqref{optfn}
become
$$
\frac{1}{t_{j-1}+1}-\frac{t_{j+1}}{(t_{j}+1)^2}=0.
$$ 
The recurrence relations
\eqref{ELeq-recur} and \eqref{Gn-rec} take the form
$$
 T_{j+1}(\tau)=\frac{(T_{j}(\tau)+1)^2}{T_{j-1}(\tau)+1},
\qquad
G_{j+1}(\tau)=G_j(\tau)+\frac{T_{j+1}(\tau)}{T_j(\tau)+1}.
%,\quad j=1,\dots,n.
$$

To make the recurrence relations more compact, we introduce the functions where the parameter and the values are shifted by $1$: 
% and transparent
$$
 \xi_j(t)=T_j(t-1)+1,
\qquad 
 \eta_j(t)=G_j(t-1)+1, \quad t\geq 1.
$$ 
Then
\begin{equation}
\label{recxi}
\xi_{j+1}(t)=\frac{\xi_{j}^2(t)}{\xi_{j-1}(t)}+1.
\end{equation}
Introduce also the auxiliary functions
\begin{equation}
\label{defalpha}
 \alpha_j(t)=\frac{\xi_j(t)}{\xi_{j-1}(t)}.
\end{equation}
%In calculations and reasoning where 
In places where $t$ does not change, we will often simply write $\xi_j$, $\eta_j$, $\alpha_j$. 

The governing system of recurrence relations becomes
\begin{equation}
\label{recrelalxi}
\begin{array}{l}
\dst
\alpha_{j+1}=\alpha_j+\frac{1}{\xi_j},
\\[2ex]
%\xi_{j}=\alpha_j \xi_{j-1}.
\xi_{j+1}=\alpha_{j+1} \xi_{j}=\alpha_{j}\xi_j+1.
\end{array}
\end{equation}
It is complemented by the subordinate recurrence
\begin{equation}
\label{recreleta}
\eta_{j+1}
%=\eta_j+\alpha_{j+1}-\frac{1}{\xi_j}
=\eta_j+\alpha_{j}.
\end{equation}

Clearly, all the introduced functions are rational functions of $t$.
Here are the first few, including the initial values ($n=0,1$) set by definition:

\bigskip\noindent
\begin{tabular}{c|c|c|c}
$n$ & $\alpha_n$ & $\xi_n$ & $\eta_n$ \\
\hline
$0$ & $t-1$    & $1$      & $1$        \\[1ex]
%\hline
$1$ & $t$ & $t$ & $t$ \\[2ex]
%\hline
$2$ & $t+t^{-1}$ & $t^2+1$ & $2t$\\[2ex]
%\hline
$3$ & $\dst\frac{t^2+1}{t}+\frac{1}{t^2+1}$ & $t^3+2t+1+t^{-1}$ & $3t+t^{-1}$\\[2ex]
%\hline
$4$ & $\cdots$ & $\dst\frac{(t^2+1)^{3}}{t^2}+\frac{2(t^2+1)}{t}+\frac{1}{t^2+1}+1$ & $\dst\frac{4t^4+6t^2+t+2}{t(t^2+1)}$\\
\end{tabular}

\bigskip 
%\medskip

Some basice consequences of the recurrence relations
\eqref{recxi}--\eqref{recreleta} are collected in the following proposition.

\begin{prop}
\label{prop:xieta-basic}
{\rm(a)} The introduced rational functions behave at inifinity as follows (for any fixed $n$): $\alpha_n(t)\sim t$, $\xi_n(t)\sim t^n$,
$\eta_n(t)\sim nt$.

\smallskip
{\rm(b)} For any $t$, the sequences $(\alpha_n(t))$, 
$(\xi_n(t))$ and $(\eta_n(t))$ are increasing.
The inequalities
$
 \alpha_n(t)\geq 2 
$ $(n\geq 1)$
and 
$
 \xi_n(t)\geq 2^{n-1}
$
$(n\geq 2)$
hold true. Also $\alpha_n(t)<3$ for all $n$.
Consequently, there exists the finite limit
$$
 \ali(t)=\lim_{n\to\infty}\alpha_n(t).
$$

\smallskip
{\rm(c)} For $n\geq 1$, define  $\Xmin_n=\min_{t\geq 1}\xi_n(t)$.
Then $\xi_n(\cdot)$ maps $[1,\infty)$ to $[\Xmin_n,\infty)$. The sequence $(\Xmin_n)$ is increasing and $\Xmin_n\geq 2^{n-1}$. 

\smallskip
{\rm(d)} For any $n\geq 1$ the following relations hold true:
\begin{equation}
\label{ic12}
\xi_n(1)=\xi_{n-1}(2), 
\qquad
\eta_n(1)=\eta_{n-1}(2),
\qquad
\alpha_n(1)=\alpha_{n-1}(2),
\end{equation}
and
\begin{equation}
\label{dxieta}
\eta'_n(t)=\frac{\xi'_n(t)}{\xi_{n-1}(t)}.
\end{equation}
%
%\smallskip
%{\rm(e)} The inequality $\xi_n'(t)/\xi_n(t)\leq n/t$ holds true for any $t\geq 1$ and $n\geq 0$.
\end{prop}

\begin{remark}
1. In (c), it is not true generally that $\xi_n(1)=X_n^0$. See Sec.~\ref{ssec:experiment}. 

\smallskip
2. To elucidate the formula \eqref{dxieta}, let us look at the dynamic programming formulation \eqref{recg}.
We have the 
objective function $\Lambda_j(x,y)=g_j(y)+L(y,x)$. Let 
$y^*=y^*(x)$ be some critical point. 
Then $\frac{\partial}{\partial y}\Lambda_j(x,y^*)=0$.
Hence
$g_{j+1}'(x)=
\frac{\partial}{\partial x}\Lambda_j(x,y^*)=L_2(y^*,x)
$. With our Lagrangian, and putting $j=n-1$, we get  $L_2(y^*,x)=L_2(T_{n-1}(\tau),T_{n}(\tau))=(T_{n-1}(\tau)+1)^{-1}$.
The result is the relation $d\eta_{n}/d\xi_{n}=\xi_{n-1}^{-1}$, equivalent to \eqref{dxieta}.
\end{remark}

\begin{proof}
(a) Immediate by induction.

(b) Monotonicity of $(\alpha_n)$ and $(\eta_n)$
is obvious from \eqref{recrelalxi} and \eqref{recreleta}. Since $\min\alpha_2=2$, the inequalities $\alpha_n> 2$ and $\xi_{n+1}>2\xi_n$ for $n>2$ follow. 

Therefore, for $n>2$ we have $\xi_n>2^{n-2}\xi_2\geq 2^{n-1}$. By \eqref{recrelalxi}, 
$\alpha_n-\alpha_2<\sum_{j=3}^{n-1} 2^{1-j}<1$.
Thus the sequence $(\alpha_n(t))$ is bounded and the %finite limit 
$\ali(t)<\infty$. 
%exists.

\smallskip
(c) We know from (b) that $\alpha_n\geq 2$ for $n\geq 2$. Now \eqref{recrelalxi} implies $\Xmin_{n+1}\geq 2\Xmin_n+1$, hence the strict monotonicity of $(\Xmin_n)$ and the estimate $\Xmin_n\geq 2^{n-1}$. 

\smallskip
(d)
%Comparing the initial conditions for $t=1$ and $t=2$, we see that $(\alpha_0(2),\xi_0(2),\eta_0(2))=(1,1,1)=
%(\alpha_1(1),\xi_1(1),\eta_1(1))$. Therefore by the recurrence relations \eqref{ic12}
%
Due to the identical initial conditions $(\alpha_0(2),\xi_0(2),\eta_0(2))=(1,1,1)=
(\alpha_1(1),\xi_1(1),\eta_1(1))$, 
the sequences $(\alpha_n(2))$ etc.\ are identical to 
$(\alpha_{n+1}(1))$ etc.

The relations \eqref{dxieta} follow by induction:
for $n=1$ we have $\eta'_1(t)=1=\xi'_1(t)/\xi_0(t)$; 
the induction step goes:
$$
 \eta_{n+1}'=\eta_n'+\left(\frac{\xi_{n+1}-1}{\xi_n}\right)'=\frac{\xi'_n}{\xi_{n-1}}+
\frac{\xi'_{n+1}}{\xi_n}-\frac{(\xi_{n+1}-1)\xi'_n}{\xi_n^2}=\frac{\xi'_{n+1}}{\xi_n},
$$
due to \eqref{recxi}.
%
\iffalse
\smallskip
(e) Put $\lambda_n(t)=\xi'_n(t)/\xi_n(t)$. From \eqref{recxi} we have the recurrence relation 
\begin{equation}
\label{logdifxirec}
 \lambda_{n+1}=(1-\xi_{n+1}^{-1})(2\lambda_n-\lambda_{n-1}).
\end{equation}
Hence%
\footnote{This is only true if $\lambda_j\geq 0$ up to $j=n+1$.}
 the sequence $(\lambda_n)$ is convex down for any $t\geq 1$. The same is true for the sequence
$\tilde\lambda_{n}=\lambda_{n}-n\lambda_1$. 
We have $\tilde\lambda_1=0$ and $\tilde\lambda_2(t)=2t/(1+t^2)-2/t<0$, hence by induction
for any $n\geq 2$
$$
 \frac{n-1}{n}\tilde\lambda_{n+1}\leq\tilde\lambda_n-
\frac{1}{n}\tilde\lambda_1<0.
\eqno\qedhere
$$
\fi
\end{proof}

\subsection{Experimental observations}
\label{ssec:experiment}

For $n=1,2,\dots$, denote by  $\gamma_n$, resp., $\gamma^T_n$, the parametric curve defined in the coordinate $(x,y)$-plane by the equations
$$
 x=\xi_n(t), \quad y=\eta_n(t), \quad t\geq 1,
$$
resp.,
$$
 x=T_n(\tau), \quad y=G_n(\tau), \quad \tau\geq 0.
$$
The curve $\gamma_n$ is obtained from $\gamma^T_n$
by the shift $(x,y)\mapsto(x+1,y+1)$.

Let $\gamma_n[1,2]$ be the part of the curve $\gamma_n$
%
%The curve $\gamma_n$ is the union of the curves 
%$\gamma_n[1,2]$ and $\gamma_n(2,+\infty)$
%
corresponding to the parameter values $1\leq t\leq 2$.
%and $t>2$, respectively.
%
Proposition~\ref{prop:xieta-basic}(d) asserts that 
$\gamma_{n-1}[1,2]$ and $\gamma_n[1,2]$ are geometrically adjacent to each other and have a common tangent at the adjacency point. 

Similarly defined partial curves $\gamma^T_n[0,1]$
%and $\gamma^T_n(1,+\infty)$ 
will be used with reference to the curves $\gamma^T_n$.

%We want to understand the relation between the curves $\gamma^T_n$ and the graphs of the functions $F_n(x)$. 

The lower envelope of the the graphs of the functions $F_n(x)$
%latter 
makes, by definition, the graph of the function $F(x)$.

One anticipates a relation between 
the graph of $F_n(x)$ 
and the curve
$\gamma^T_n$ 
based on the premise that the abscissa $x=T_n(\tau)$ corresponds to the value of $\tau$ that determines the extremal trajectory yielding the extremal value $y=G_n(\tau)$ of the action. We take this view as a ``first approximation''
and discuss necessary corrections below.

The curve $\gamma^T_n$ lies in the half-plane
$x\geq\Xmin_n-1$ (so that the equation $T_n(\tau)=x$ has a solution).
Consequently, $\gamma^T_n$ cannot represent the whole
of the graph of $F_n(x)$
(which is defined for all $x>0$). 
Parts of the graph will be represented by segments
of curves $\gamma^T_k$, $k<n$; we elaborate on this 
in \S~\ref{ssec:bndmin}.

\begin{figure}[htb]
\begin{picture}(260,266)
\put(0,0){\includegraphics[scale=0.5]{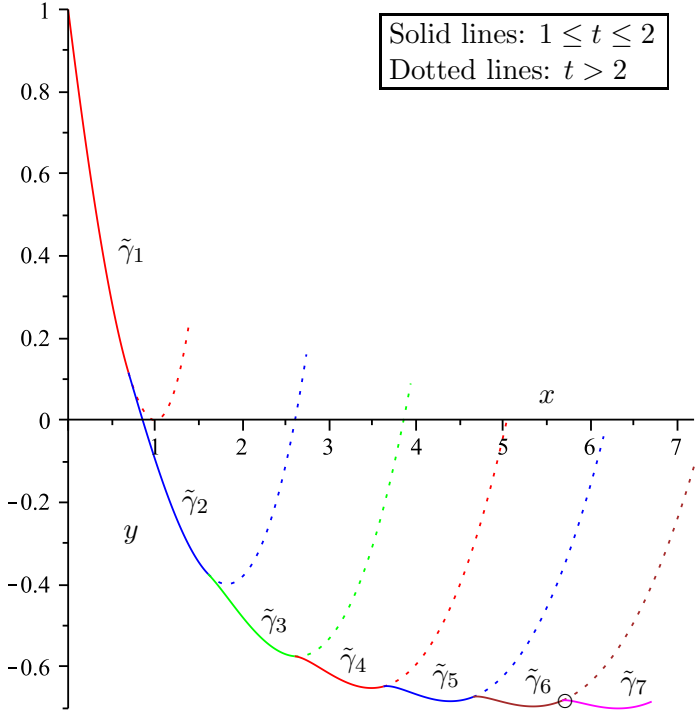}
}
% Origing: file xieta.mw
%\put(215,120){\small $x$}
%\put(44,65){\small $y$}
%\put(56,175){\small $\tgamma_1$}
%\put(80,80){\small $\tgamma_2$}
%\put(110,35){\small $\tgamma_3$}
%\put(140,20){\small $\tgamma_4$}
%\put(175,14){\small $\tgamma_5$}
%\put(220,14){\small $\tgamma_6$}
%\put(155,250){\framebox{\parbox{3.6cm}{\small {Solid lines: $1\leq t\leq 2$}\\
%{Dotted lines: $t> 2$}}}}
\put(201,116){\small $x$}
\put(44,65){\small $y$}
\put(42,171){\small $\tgamma_1$}
\put(66,76){\small $\tgamma_2$}
\put(96,31){\small $\tgamma_3$}
\put(126,16){\small $\tgamma_4$}
\put(161,10){\small $\tgamma_5$}
\put(196,8){\small $\tgamma_6$}
\put(232,7.4){\small $\tgamma_7$}
\put(141,246){\framebox{\parbox{3.6cm}{\small {Solid lines: $1\leq t\leq 2$}\\
{Dotted lines: $t> 2$}}}}
\put(211,3){\circle{5}}
\end{picture}
\caption{Curves $\tgamma_n$, $n=1,\dots,7$,
with parametric equations $x=\log\xi_n(t)$, $y=\eta_n(t)-e x$. Inside the little circle is a region displayed in Fig.~\ref{fig:xieta6-7}.} 
\label{fig:xieta1-6}
\end{figure}

The mutual position of the curves $\gamma_n$ (or $\gamma^T_n$) with $1\leq n\leq 7$ is illustrated in Figure \ref{fig:xieta1-6}.  
%in the $x$-logarithmic scale.
As a matter of fact, shown are the curves $\tgamma_n$ obtained from $\gamma_n$ by the transformation $(x,y)\mapsto (\log x,y-e\log x)$ (scaling conveniently and removing the asymptotic 
%linear 
drift).
%The parts depicted by solid lines correspond to the parameter values $1\leq t\leq 2$ and dotted parts correspond to some subintervals $2<t<t^{(\rm max)}_n$.
The %envelope of the solid curves, i.e.\ the 
union of the solid parts forms a part of the graph of the function $F(e^{x}-1)+1-ex$. %=\min_n F_n(x)$.

%One is tempted to claim 
Figure~\ref{fig:xieta1-6} seems to support the view that: (a) the parametric curves $\gamma^T_n$ are the graphs of the functions %$F_n(e^x-1)+1-e x$ 
$F_n(x)$
restricted to $x\geq \Xmin_n-1$; 
(b) moreover, $\Xmin_n=\xi_n(1)$; and (c) the lower envelope of the curves 
$\gamma_n$
%({\em mutatis mutandis}, the curves $\dgamma_n$ in Fig.~\ref{fig:xieta1-6}),
%which is the graph of the function $f(x-1)+1=1+\min_n f_n(x-1)$, 
coincides with union of the segments $\gamma_n[1,2]$. 

\begin{figure}[ht]
\begin{picture}(260,255)
\put(0,0){\includegraphics[scale=0.5]{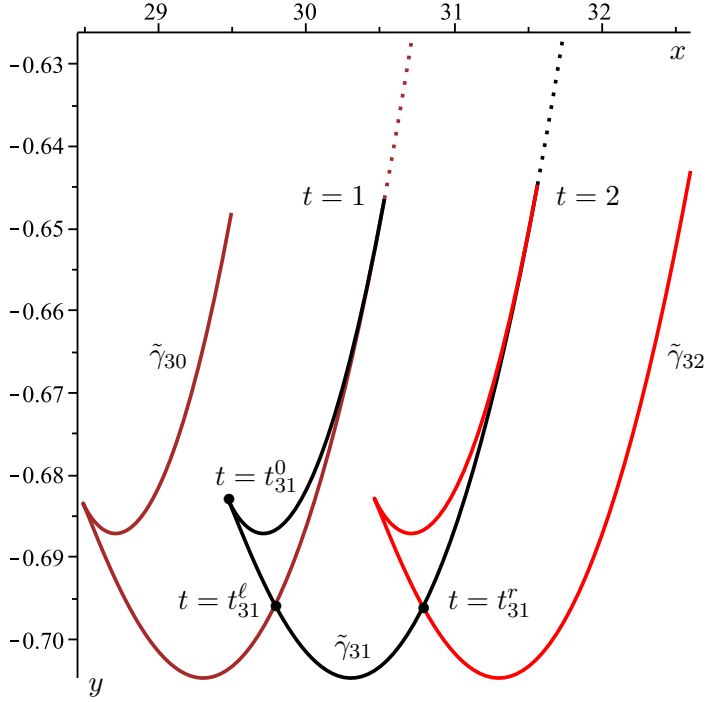}
}
% Origing: file xieta.mw
%\put(265,240){\small $x$}
%\put(44,1){\small $y$}
%\put(67,125){\small $\tgamma_{30}$}
%\put(137,15){\small $\tgamma_{31}$}
%\put(263,125){\small $\tgamma_{32}$}
%\put(125,185){\small $t=1$}
%\put(221,185){\small $t=2$}
%\put(97.7,73.7){\circle*{4}}
%\put(92,79){\small $t=\Tmin_{31}$}
%\put(115.5,33){\circle*{4}}
%\put(78,32){\small $t=\Tleft_{31}$}
%\put(171.47,32.41){\circle*{4}}
%\put(180,32){\small $t=\Tright_{31}$}
%
\put(250,235){\small $x$}
\put(30,-4){\small $y$}
\put(53,120){\small $\tgamma_{30}$}
\put(123,10){\small $\tgamma_{31}$}
\put(249,120){\small $\tgamma_{32}$}
\put(111,180){\small $t=1$}
\put(207,180){\small $t=2$}
\put(83.2,68.7){\circle*{4}}
\put(78,74){\small $t=\Tmin_{31}$}
\put(101,28){\circle*{4}}
\put(64,27){\small $t=\Tleft_{31}$}
\put(156.9,27.41){\circle*{4}}
\put(166,27){\small $t=\Tright_{31}$}
\end{picture}
\caption{Parametric curves $\tgamma_n$, $n=30,31,32$.
The coordinates and legend are the same as in Fig.~\ref{fig:xieta1-6}. The $t$-value marks pertain to $\tgamma_{31}$.} 
\label{fig:xieta30-32}
\end{figure}

These impressions 
%turn out to be deceptive. They 
are refuted by observing the curves with greater values of $n$.
Figure~\ref{fig:xieta30-32} shows that:
(a) the curve $\tgamma_n$, and hence $\gamma_n$ or $\gamma^T_n$, in general is not a graph of a single-valued function;
(b) the value $\Xmin_n$ is the abscissa of the cusp
and corresponds to some $\Tmin_n\in (1,2)$;
(c) the lower envelope of the
curves $\gamma_n$ 
is a proper subset of the
%does not coincide with 
union of the segments $\gamma_n[1,2]$.

\begin{figure}[thb]
\begin{picture}(257,255)
\put(0,0){\includegraphics[scale=0.5]{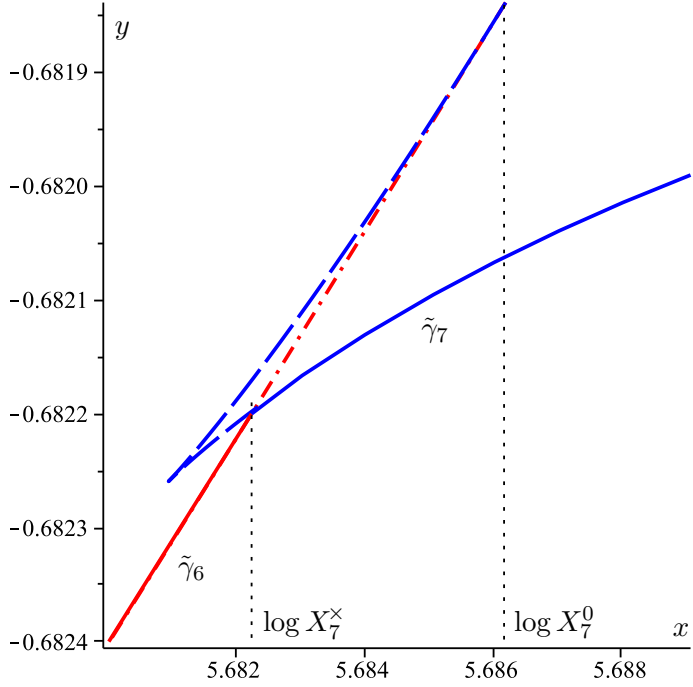}
}
% Origing: file xieta.mw
%\put(265,32){\small $x$}
%\put(54,260){\small $y$}
%\put(78,55){\small $\tgamma_6$}
%\put(170,145){\small $\tgamma_7$}
%\put(110,34){\small $\log\Xcross_7$}
%\put(205,34){\small $\log\Xmin_7$}
%
\put(251,16){\small $x$}
\put(40,244){\small $y$}
\put(64,39){\small $\tgamma_6$}
\put(156,129){\small $\tgamma_7$}
\put(96,18){\small $\log\Xcross_7$}
\put(191,18){\small $\log\Xmin_7$}
\end{picture}
\caption{%An enlarged 
A magnified
view of the curves
$\tgamma_6$ in parameter region $t\to 2^-$ and $\tgamma_7$
in parameter region $t\to 1^+$. Solid lines form a part of the graph of
$F(x)$.
} 
\label{fig:xieta6-7}
\end{figure}

Let us describe what {\em is}\ true and will be proved in the sequel (Sec.~\ref{ssec:special-points}).

\smallskip
(i) Case $n\leq 6$. The observations based on Fig.~\ref{fig:xieta1-6} are mostly adequate (with subtle exception described in (iii) below), the crucial fact being that the functions $\xi_n(t)$ and $\eta_n(t)$ are monotone increasing in $t\in[1,\infty)$.
%; the curve
%$\gamma^T_n$ is the graph of the function $f_n(x)$
%restricted to $x\geq \Xmin_n-1$, and %$\Xmin_n=\xi_n(1)$.
%As follows from Proposition~\ref{prop:xieta-basic},
At $x=T_n(0)=T_{n-1}(1)=\Xmin_n-1$, we have $F(x)=G_n(0)=G_{n-1}(1)=\eta_n(1)-1$; 
the function $F(\cdot)$ has a continuous derivative at that point. 

\smallskip
(ii) Case $n\geq 7$. There exist three special values of the parameter: $1<\Tmin_n<\Tleft_n<\Tright_n<2$.
The functions $\xi_n(t)$ and $\eta_n(t)$ are convex in $t\in[1,\infty)$;
they decrease in $[1,\Tmin_n]$ and increase in $[\Tmin_n,\infty)$.
The value $\Xmin_n=\min_t \xi_n(t)$ equals $\xi_n(\Tmin_n)$.
The curve $\gamma_n$ has a cusp at the point
$(\Xmin_n,\eta_n(\Tmin))$.
The lower branches of the curves $\gamma_n$ and
$\gamma_{n-1}$ meet at the point with coordinates
$(\xi_n(\Tleft_n),\eta_n(\Tleft_n))=
(\xi_{n-1}(\Tright_{n-1}),\eta_{n-1}(\Tright_{n-1}))$.
Set by definition $\Xcross_n=\xi_n(\Tleft_n)$.
So $\Xcross_{n+1}=\xi_n(\Tright_n)$.

The branch of $\gamma_n$ corresponding to parameter values $t\geq\Tleft_n$ is the graph of the function $F_n(x-1)+1$ restricted to $x\geq \Xcross_n$.
(As to why $\Xcross_n$ and not $\Xmin_n$ here -- see the end of \S~\ref{ssec:bndmin}.)
The point $(T_n(0),G_n(0))$ does not belong to the lower branch of the curve $\gamma_n$.
The equality
$F(x)=F_n(x)$ is valid for $x\in[\Xcross_n,\Xcross_{n+1}]$.
 The left and right derivative numbers of $f(x)$
at $x=\Xcross_n$ are not equal to each other. 

\smallskip
(iii) The exception mentioned in (a) applies to 
the case $n=6$.
Looking at Fig.~\ref{fig:xieta1-6},
one is led to mistakenly believe that the equality
$F(x)=F_6(x)$ holds true for $\xi_6(1)\leq x\leq \xi_6(2)$,
while in fact it is valid in the shorter interval
$\xi_6(1)\leq x\leq\Xcross_7$, see Figure~\ref{fig:xieta6-7}. Let us clarify that
the value $n=7$ is the first for which $\Xcross_n$ is defined,
since $\gamma_7$ is the first curve having a cusp.
The value $\Tleft_6$ is not defined, but $\Tright_6\approx 1.9975$
is defined and 
$\Xcross_7=\xi_7(\Tleft_7)=\xi_6(\Tright_6)$.

\subsection{Case of a minimizer at the boundary}
\label{ssec:bndmin}

We will show that if the minimum of the objective function
$\hat S_n(\mathbf{t})$ is attained at the boundary of the admissible region, then the minimizer corresponds to 
an extremal trajectory of a shorter length.

%Suppose $x<\Xmin_n=\min\xi_n(t)$, so that the equation $\xi_n(t)=x$ does not have a solution.
%This means that the function $S_n(0,t_1,\dots,t_{n-1},x)$ does not have a critical point on the octant $t_i>0$, $i=1,\dots,n-1$. Therefore, if $\mathbf{t^*}$ is the minimizer (it exists by Proposition~\ref{prop:monotfn}(d)), then
%$t^*_j=0$ for at least one $j\in\{1,\dots,n-1\}$.

\begin{lemma}
\label{lem:min-at-bnd}
(a) Suppose $\mathbf{t^*}=(0,t^*_1,\dots,t^*_{n-1},x)$ is a minimizer for the problem \eqref{optfn}.
If $t^*_j=0$ for some $j\in\{1,\dots,n-1\}$, then
$t^*_i=0$ for any $i\in\{1,\dots,j-1\}$.
In this case $F_n(x)=F_{n-j}(x)$.

(b) If, %at the same time, 
in addition, 
$t^*_{j+1}\neq 0$,
then the truncated vector
$$\hat{\mathbf{t}}^*=(0=t_{j}^*,t_{j+1}^*,\dots,t_{n-1}^*,x)
$$
is an extremal trajectory for the problem \eqref{optfn}
with $n$ replaced by $n-j$. In this case $T_{n-j}(t^*_{j+1})=x$ and $G_{n-j}(t^*_{j+1})=F_n(x)$.
\end{lemma}

%---------------------------------------------------------------
\iffalse
\begin{prop}
\label{prop:fnstab}
For every $n\geq 1$ there exists $X_n$ such that
$f_{n+1}(x)=f_{n}(x)$ for $0\leq x\leq X_n$.
For instance, $X_1=1$.
\end{prop}

\begin{proof}
Suppose the $n$-tuple $(t_1,\dots,t_{n}$ 
minimizes $S_{n-1}(t_1,\dots,t_n)$ under constraint
$t_n=x$. If 
$f_{n+1}(x)<f_{n}(x)$, then 
there exists an $(n+1)$-tuple $t'_0,\dots,t'_{n-1},t_n$ with $t'_n=x$ such that
$$
S_n(t'_0,\dots,t'_{n-1},x)<f_{n-1}(x)=S_{n-1}(t_1,\dots,t_n).
$$

We need to show that for some $X_n$
and any $t_1,\dots,t_n$ with $t_n\leq X_n$ the inequality
$$
 S_n(t_0,\dots,t_n)\leq S_{n-1}(t_1,\dots,t_n)
$$
implies $t_0=0$.
\end{proof}
\fi
%---------------------------------------------------------------

\begin{proof}
(a) Put $i=j-1$. If $i>0$ and $t^*_i>0$, then the necessary condition of
extremum $\partial \hat S_n/\partial t_i|_{\mathbf{t}=\mathbf{t}^*}=0$
yields
$$
\frac{1}{t_{i-1}^*+1}-\frac{t_{j}^*}{(t_{i}^*+1)^2}=0,
$$ 
which contradicts the assumption $t_j^*=0$. Hence $t_{j-1}^*=0$. By downward induction on $i$ we get $t_i=0$ for any $i<j$. 

Since
$$
 \hat S_n(0,\dots,0,t_{j+1},\dots,x)=S_{n-j}(0,t_{j+1},\dots,x),
$$
we have $F_n(x)=\hat S_{n-j}(\hat{\mathbf{t}}^*)\geq F_{n-j}(x)$. It follows that $F_n(x)=F_{n-j}(x)$.

\smallskip
(b) For the same reason, $\hat{\mathbf{t}}^*$ is the minimizer for
\eqref{optfn}
with $n$ replaced by $n-j$. Moreover,
$t^*_i>0$ for $i=j+1,\dots,n$. Hence $\hat{\mathbf{t}}^*$ is an extremal trajectory of length $n+1-j$ and $T_{n-j}(t_{j+1}^*)=x$.
%
%$t=t^*_{j+1}$ is a root of the equation
%$T_{n-j}(t)=x$.
Finally, $\hat S_{n-j}(\hat{\mathbf{t}}^*)=G_{n-j}(t^*_{j+1})$,
hence $F_n(x)=G_{n-j}(t^*_{j+1})$.
\end{proof}

We see that a minimizer that lies on the boundary of
the admissible domain is represented by an extremal trajectory of length $k\in\{2,\dots,n\}$. 
Consequently, the graph of $F_n(x)$ 
is the lower envelope of the parametric curves 
$\gamma^T_k$ with $1\leq k\leq n-1$.

For $n\leq 6$ and $1\leq k\leq n-1$, the graph of the restriction $F_n(x)|_{[T_k(0),T_k(1)]}$ coinsides with
$\gamma^T_k[1,2]$.

For $n\geq 7$, the part of the curve $\gamma^T_n$ corresponding to parameter values $\Tmin_n<t<\Tleft_n$
lies above the curve 
$\gamma^T_{n+1}$ and hence does not belong to the
graph of $F_n(x)$. 

%===========================================
%===========================================

\subsection{Monotonicity of the critical index}
\label{ssec:crind}

\begin{definition}
\label{def:crind}
For the given $x>0$, the {\em critical index} $\nu(x)$ is the integer equal to minimum value of $n$ such that $F(x)=F_{n}(x)$.
\end{definition}

\begin{prop}
\label{prop:crind}
The function $x\mapsto \nu(x)$ is nondecreasing.
\end{prop}

\begin{proof}
By the definition of extremal trajectories and in view ofLemma~\ref{lem:min-at-bnd}, the equality $n=\nu(x)$ holds if and only if both of the following
are true:

\smallskip
(i) there exists $\tau>0$ 
 such that $T_n(\tau)=x$ and $G_n(\tau)=F(x)$;
 equivalently, the vector $\mathbf{t}\in\RR^{n+1}$ with 
 $t_0=0$ and $t_{j}=T_j(\tau)$
($j=1,\dots,n)$ is a minimizer for the problem \eqref{optfn};

\smallskip
(ii) for any $k<n$ either the equation $T_{k}(\tau)=x$ does not have a solution or, if 
$\tau$ is a solution, then $G_{k}(\tau)>F(x)$.  

\smallskip
Consequently, if $\mathbf{t}$ is an extremal trajectory specified in (i), then $\nu(t_j)=j$ ($1\leq j\leq n$).  

Using this observation, we will show  by induction on $n$ that the inequality $n=\nu(x)>\nu(x')$ implies $x>x'$.

Let us assume that the said implication is true with $n-1$ instead of $n$. (The special case $n-1=1$ is included.) 

Suppose that the induction step fails.
It means that there exist $x$ and $x'$ such that $x<x'$ and $n=\nu(x)>\nu(x')=k$.

Let $\mathbf{t}=(0,t_1,\dots,t_{n-1},x)$ and
$\mathbf{t'}=(0,t'_1,\dots,t'_{k-1},x')$
be the extremal trajectories with $G_n(t_1)=F(x)$
and $G_k(t'_1)=F(x')$.
Then $\nu(t_{n-1})=n-1$ and $\nu(t'_{k-1})=k-1$. By the inductive assumption, $t'_{k-1}<t_{n-1}$.
(If $n=2$, then this is a tautology: $0=t'_0<t_1$.)
We have
$$
 F(x')=S_k(\mathbf{t'})=\hat S_k(0,t'_1,\dots,t'_{k-1},x)
+\frac{x'-x}{1+t'_{k-1}}.
$$
By the assumption of {\em ad absurdum}\ argument, $k<n$, so
$$
\hat S_k(0,t'_1,\dots,t'_{k-1},x)\geq F_k(x)>F(x)
=\hat S_n(0,t_1,\dots,t_{n-1},x).
$$
Due to the inequality $t'_{k-1}<t_{n-1}$, 
$$
 \frac{x'-x}{1+t'_{k-1}}>\frac{x'-x}{1+t_{n-1}}.
$$
Therefore
$$
 F(x')> \hat S_n(0,t_1,\dots,t_{n-1},x)
+\frac{x'-x}{1+t_{n-1}}=\hat S_n(0,t_1,\dots,t_{n-1},x'),
$$
which contradicts the assumption that $\mathbf{t}$
is a minimizer for $F(x')$.
\end{proof}

\section{Reduction of Theorem~\ref{thm:main} to  Theorem~\ref{thm:paramcurves}}
\label{sec:proofmain}

In the course of the proof, which involves many small technical steps, we will prove that Fig.~\ref{fig:xieta30-32} adequately illustrates the relevant features of the curves $\gamma_n$ 
(defined in Sec.~\ref{sec:minimizers})
with large enough $n$.

We will explore in great detail a parametrization of the curves $\gamma_n$. As a result, we will be able to describe
a piece-wise parametrization of their lower envelope
and to derive the asymptotics of the function $f(x)$.

%The visual picture, supported by experimental evidence, has been outlined above.  

\subsection{Only the partial curves \texorpdfstring{$\gamma_n[1,2]$}{gamma[1,2]} are relevant}
\label{ssec:gamma12}

Observing the dotted lines in Figs.~\ref{fig:xieta1-6} and \ref{fig:xieta30-32},
one is led to conjecture that the parts of the curves $\gamma_n$ %(2,+\infty)$
corresponding to the parameter values $t>2$
do not contribute to the lower envelope of the
curves $\gamma_n$. Equivalently, 
the parts of the curves $\gamma^T_n$ %(1,+\infty)$ 
corresponding to the parameter values $\tau>1$
do not contribute to the graph of the function $F(x)$.
We will prove this conjecture now.

As shown in \S~\ref{ssec:bndmin}, the coordinates of any point of the graph of $F(x)$ can be written as
$x=T_n(\tau)$, $F(x)=G_n(\tau)$ with some $n\geq 1$
and $\tau>0$. 

\begin{prop}
\label{prop:tau01}
Given $x>0$, suppose that $n$ and $\tau$ are such that $x=T_n(\tau)$ and $F(x)=F_n(x)=G_n(\tau)$. Then $\tau\leq 1$.
\end{prop}

\begin{proof}
We have
$$
 F_n(x)=\hat S_n(0,t^*_1,\dots,t^*_n),
$$
where $t^*_1=T_1(\tau)=\tau$ and $t^*_n=T_n(\tau)=x$.

Suppose, contrary to what is claimed, that $\tau>1$.
Define a vector $\mathbf{t}=(0,t_1,\dots,t_{n+1})
\in\RR_+^{n+2}$ as follows:
$$
 t_1=s, \quad\text{and}\quad
 t_{j+1}=t^*_j \;\; (j=1,\dots,n),
$$
where $s$ is an arbitrary number such that $0<s<\tau-1$.
Then %(cf.~\eqref{optfn})
$$
 \hat S_n(\mathbf{t^*})-\hat S_{n+1}(\mathbf{t})
=t^*_1-\left(t_1+\frac{t_2}{t_1+1}\right)=
\tau-\left(s+\frac{\tau}{s+1}\right)=\frac{s(\tau-1-s)}{s+1}>0,
$$
so $F(x)\leq F_{n+1}(x)\leq \hat S_{n+1}(\mathbf{t})<\hat S_n(\mathbf{t^*})=F_n(x)$, a contradiction. 
\end{proof}

\subsection{Convexity of the functions \texorpdfstring{$\alpha_n(t)$ and $\alpha_\infty(t)$,
$t\in[1,2]$}{alpha[n] on [1,2]}}
\label{ssec:alpha}

Recall that the functions $\alpha_j(t)$ is defined as
$\alpha_j(t)=\xi_j(t)/\xi_{j-1}(t)$.
They are a part of the recurrent scheme \eqref{recrelalxi} defining extremal trajectories.
The function $\alpha_\infty(t)=\uparrow\lim_{j\to\infty}\alpha_j(t)$ determines the eventual rate of the exponential 
growth of the components of an extremal trajectory
starting at $\tau=t-1$. The graph of the function
$\alpha_\infty(t)$ on the interval $[1,2]$ is shown in
Fig.~\ref{fig:alphaplot}. Crusial for the derivation of the asymptotic formulas is the existence of the solution
of the equation $\alpha_\infty(t)=e$ in $[1,2]$.
The roots are marked $t_a$ and $t_b$ on the figure.

\begin{prop}
\label{prop:alpha}
The function $\alpha_\infty(t)$ is real-analytic in $[1,2]$, convex, and has one point of minimum at $t_o\in(1,2)$.
The inequality
$$
\alpha_\infty(t_o)<e<\alpha_\infty(1)=\alpha_\infty(2)
$$
holds. 
Consequently, there are uniquely determined $t_a$, $t_b$ with $1<t_a<t_o<t_b<2$ and $\alpha_\infty(t_a)=\alpha_\infty(t_b)=e$.

Also the functions $\alpha_n(t)$ are convex.
\end{prop}

\begin{figure}[htb]
\begin{picture}(260,255)
\put(0,0){\includegraphics[scale=0.5]{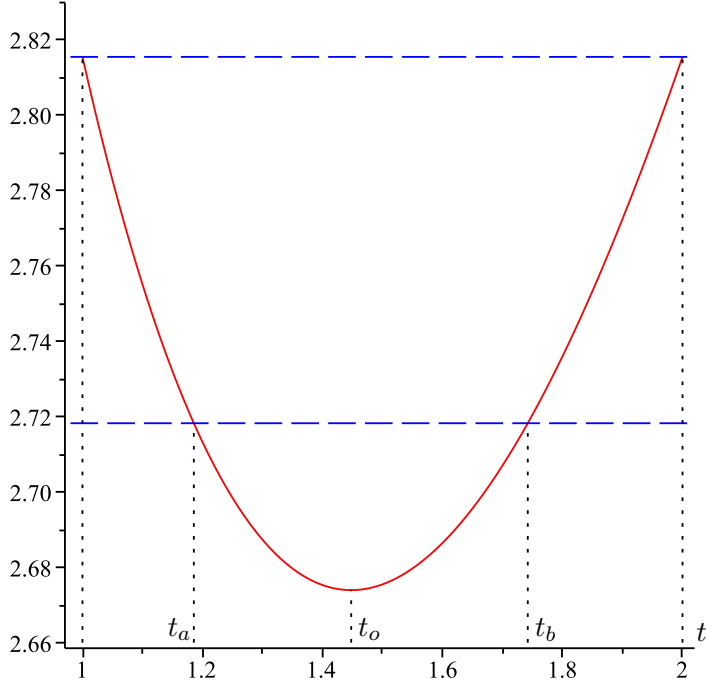}
}
% Origing: file xieta.mw
\put(260,14){\small $t$}
\put(60,16){\small $t_a$}
\put(132,16){\small $t_o$}
\put(199,16){\small $t_b$}
\end{picture}
\caption{Function $\alpha_\infty(t)$ on the interval $[1,2]$} 
\label{fig:alphaplot}
\end{figure}

In \eqref{num-alpha} the numerical values are given for reference. As a matter of fact, only
$t_b$ will be relevant in the proof of Theorem~\ref{thm:main}.

\begin{equation}
\label{num-alpha}
\ba{l}
\alpha_\infty(1)=\alpha_\infty(2)\approx 2.815572650,
%2.815572649889378477459567764887611469473
\\[0.7ex]
%The point of minimum: [checked 7 digit accuracy - Maple {\verb!funceq_iter.mw!}]
t_o\approx 1.447847, 
\quad 
\alpha_\infty(t_o)\approx 2.673953412
\\[0.7ex] 
t_a\approx 1.185591828, %796,
 \quad
t_b\approx 1.742084284. %155.
\ea
\end{equation}

The behaviour of $\alpha_\infty$ is simple, but to prove it analytically is not an easy task.
Our proof is rather long and technical. It is given in Appendix~\ref{app:alpha}.

\smallskip
The next lemma, which is a simple corollary
of Proposition~\ref{prop:alpha}, will be important in
the proof of Proposition~\ref{prop:Xmin}(b) below.

\begin{lemma}
\label{lem:conv-xi}
The functions $\xi_n(t)$, $n=1,2,\dots$, are convex
in $t\in[1,2]$.
\end{lemma}

\begin{proof}
The claim is true for $n=1$, since $\xi_1(t)=t$.
The general case follows by induction due to
the identity $\xi_{n+1}(t)=\alpha_n(t)\xi_n(t)+1$
and the convexity of $\alpha_n(t)$. 
\end{proof}

\subsection{Special points on the curves \texorpdfstring{$\gamma_n[1,2]$}{gamma[n]}}
\label{ssec:special-points}

Recall that the special values $\Xmin_n=\min_{t\geq 1}\xi_n(t)$ 
have been defined in Proposition~\ref{prop:xieta-basic}(c).
Thus $\Xmin_n$ is such a value that the equation
\beq{Tnx}
T_n(\tau)=x
\eeq
has a solution if and only if $x\geq\Xmin_n-1$.

In Sec.~\ref{ssec:experiment} we have identified and illustrated (Fig.~\ref{fig:xieta30-32}) special values of parameter $t$: $\Tmin_n$, $\Tleft_n$, and $\Tright_n$.
Here we define them analytically. 
($\Tleft_n$ is defined for $n\geq 7$.) For the curves $\gamma_n^T$, the corresponding values of $\tau=t-1$  will be denoted similarly: $\tau^{0}_n$, $\tau^{\ell}_n$, $\tau^{r}_n$.

\begin{prop}
\label{prop:Xmin}
{\rm(a)} $\xi_n'(t)>0$ for all $t\geq 2$ and $n\geq 1$.
Consequently, the solution $t=\Tmin_n$ of the equation  $\xi_n(t)=\Xmin_n$, belongs to $[1,2)$.
Equivalently, the solution $\tau_n^0$ of the equation 
\eqref{Tnx} with $x=\Xmin_n-1$ lies in $[0,1)$.

\smallskip
{\rm(b)} For every $n\geq 1$ and every $x\geq \Xmin_n-1$
the equation~\eqref{Tnx} has at most two solutions. 

\smallskip
{\rm(c)} The sequences $(t_n^0)$ (hence $(\tau_n^0)$) and $(\Xmin_n)$ are nondecreasing. More precisely, 
$t_n^0=1$ for $n\leq 6$ and $t_n^0>t_{n-1}^0$ for $n\geq 7$. 
\end{prop}

\begin{definition}
\label{def:t_pm}
If the equation~\eqref{Tnx} has two distinct solutions,
the smaller will be denoted $\tau_n^-(x)$ and the larger $\tau_n^+(x)$. 

If there is a unique solution, it will be denoted $\tau_n^+(x)$. In this case we leave $\tau_n^-(x)$ undefined, except when $x=\Xmin_n-1$. We put $\tau_n^-(\Xmin_n-1)=\tau_n^+(\Xmin_n-1)=\tau_n^0$.  

Similar notation $t_n^{\pm}(x)$ will be used in reference to the equation $\xi_n(t)=x$.
\end{definition}

\begin{proof}
(a) Since $\xi_n=\alpha_n\alpha_{n-1}\dots\alpha_1$,
it suffices to prove that $\alpha'_n>0$ for
$t\geq 2$ and all $n\geq 1$. We take this inequality
as the induction hypothesis. It is true for $n=1$.

We have $\alpha_n=\alpha_1+\sum_{j=1}^{n-1}\xi_j^{-1}$
and $\alpha_1=t$.
We need to prove that
$$
 \sum_{j=1}^{n-1}\frac{\xi'_j}{\xi_j^2}< 1=\alpha_1'.
$$
As a consequence of the induction hypothesis, 
$\xi'_k>0$ for $1\leq k\leq n-1$ and $t\geq 2$.
Hence 
$\alpha'_j<\alpha'_1$, so $0<\xi'_j/\xi_j=\sum_{k=1}^j
\alpha'_k/\alpha_k<j/t\leq j/2$ for $j=1,\dots,n-1$.
Also, $\xi_j\geq\alpha_1^j\geq 2^j$. The required estimate follows:
$$
 \sum_{j=1}^{n-1}\frac{\xi'_j}{\xi_j^2}\leq
\sum_{j=1}^{n-1}\frac{j}{2^{j+1}}<\sum_1^\infty\frac{j}{2^{j+1}}=1.
$$

\smallskip
(b) The functions $\xi_n(t)$, are convex in $1\leq t\leq 2$ by Lemma~\ref{lem:conv-xi}; more precisely, for $n\geq 2$ they are strictly convex. They are increasing in $[2,\infty)$ by part (a).
Therefore the equation $\xi_n(t)=x$ has at most two solutions.

\smallskip
(c) Due to the convexity of $\xi_n(\cdot)$, there are two possibilities: either (i) $t^0_n=1$ and $\Xmin_n=\xi_n(1)$ or (ii) $t^0_n>1$, then $\xi'(t^0_n)=0$ and the function $\xi(\cdot)$ decreases from $\xi_n(1)$ to $\Xmin_n$ as $t$ changes from $1$ to $t^0_n$. (Cf.\ the backtracking segment of the curve $\tilde\gamma_7$ in Fig.~\ref{fig:xieta6-7}.)

In both cases, $\xi_n'(t)>0$ for $t>t_n^0$.
%if $\xi_n'(t_*)\geq 0$ for some $t_*$, then
%$\xi_n'(t)> 0$ for all $t>t_*$. (In other words, the sign %of $\xi_n'$ cannot change from $+$ to $-$.)

Suppose that the sequence $(t_n^0)$ is not monotone as claimed. Let $n$ be the least index for which $t^0_n>t^0_{n+1}$. Then $\xi'_{n+1}(t^0_n)<0$ and $\xi'_n(t^0_n)=0$. Also
 $t^0_{n-1}\leq t^0_n$, so $\xi'_{n-1}(t^0_n)\leq 0$. 
The combination of signs $\xi'_{n+1}<0$, $\xi'_n=0$,
$\xi'_{n-1}\leq 0$ contradicts the recurrence relation
$$
 \frac{\xi'_{n+1}}{\xi_{n+1}-1}=\frac{2\xi'_n}{\xi_n}-\frac{\xi'_{n-1}}{\xi_{n-1}},
$$ 
which follows from \eqref{recxi} by logarithmic differentiation.

By the recurrence relations \eqref{recxi} and the one above we find that $\xi'_n(1)>0$ for $1\leq n\leq 6$,
while $\xi'_7(1)=-\frac{19661554943536}{328636389375}<0$. 
Therefore $t^0_n>1$ for $n\geq 7$.

\smallskip
For any $t$, by definition of $\Xmin_n$, we have
$\Xmin_n\leq \xi_n(t)$. Since $\xi_n(t)<\xi_{n+1}(t)$,
it follows that
$\Xmin_n\leq\inf_{t\geq 1}\xi_{n+1}(t)=\Xmin_{n+1}$.
It is easy to see that the inequality is strict when 
$t^0_n<t^0_{n+1}$, i.e. for $n\geq 6$.
\end{proof}

Next we will analytically prove that the branches of the curve $\gamma_n$ are always situated as shown in  Fig.~\ref{fig:xieta30-32} or Fig.~\ref{fig:xieta6-7}:
the ``backward'' branch (where $t<t^0_n$)
lies above the ``forward'' branch. 

\begin{prop}
\label{prop:pmbranches}
If the equation $\xi_n(t)=x$ has two distinct solutions, then $\eta_n(t_n^+(x))<\eta_n(t_n^-(x))$.
Equivalently,
%if $\tau_n^+(x)$ and $\tau_n^-(x)$ are defined and %$x\neq\Xa_n$, then
$G_n(\tau_n^+(x))<G_n(\tau_n^-(x))$.
\end{prop}

\begin{proof}
The equation $\xi_n(t)=y$ has two distinct solutions for
all $y\in(\Xmin_n,x]$. For $y=\Xmin_n$ the value
$t_n^-(y)=t^0_n=t_n^+(y)$ is also defined. Put
$$
 \kappa(y)=\xi_{n-1}(t_n^{+}(y))-\xi_{n-1}(t_n^{-}(y)).
$$

Clearly, $\kappa(\Xmin_n)=0$.
We will prove that: 
(i) $\kappa(\Xmin_n+\eps)>0$
for sufficiently small $\eps>0$, 
and 
(ii) $
 \kappa(y)\neq 0 
$
for any $y\in(\Xmin_n,x]$.
By continuity of $\kappa(\cdot)$ it follows then that
$\kappa(x)>0$. 

\smallskip
Proof of (i): Due to the inequality $t_n^0>t_{n-1}^0$ (Proposition~\ref{prop:Xmin}(c)), the function 
$\xi_{n-1}(t)$ increases in the neighborhood of $t^0_n$.

\smallskip
Proof of (ii): Suppose, by way of contradiction, that 
$\kappa(y_*)=0$ for some $y_*>\Xmin_n$. 
Let $t_*^\pm=t_n^\pm(y_*)$.
Then 
$y=\xi_{n}(t_*^-)=\xi_{n}(t_*^+)$
and $\xi_{n-1}(t_*^-)=\xi_{n-1}(t_*^+)$.
By the inverse recurrence relation we conclude that
$\xi_{j}(t_*^-)=\xi_{j}(t_*^+)$ for $j$ from $n$ down to $1$. For $j=1$ it results in $t_*^-=t_*^+$,
a contradiction.

Now, by \eqref{dxieta}
$$
 \frac{d\eta_n(t^{\pm}(y))}{d y}=\frac{1}{\xi_{n-1}(t^{\pm}(y))}.
$$
The inequality $\kappa(y)>0$ implies
$$
 \frac{d\eta_n(t^{+}(y))}{d y}<
\frac{d\eta_n(t^{-}(y))}{d y}.
$$
Integrating from $y=\Xmin_n$ to $x$ we get $\eta_n(t^{+}(x))<\eta_n(t^{-}(x))$.
\end{proof}

\begin{definition}
\label{def:xcross}
For $n\in\NN$, let $\Xcross_n$ be defined by%
\footnote{The subtraction of $1$ 
 makes this definition consistent with notation $\Xcross_n$ in Sec.~\ref{ssec:experiment}.}
$\Xcross_n-1=\inf\{x\mid\nu(x)=n\}
=\max\{x\mid\nu(x)=n-1\}$, where $\nu(\cdot)$
is the critical index (Definition~\ref{def:crind}).
%Sec.~\ref{ssec:crind}
\end{definition}
%The function $\Xcross_{\cdot}$ is the maximal right-inverse to $\nu(\cdot)+1$.

\begin{prop}
\label{prop:argmin_intervals}
{\rm (a)} $F(x)=F(x_{n-1})$ for $0<x\leq \Xcross_n$
and  $F(x)<F(x_{n-1})$ for $x>\Xcross_n-1$.
The point $(\Xcross_n-1,\,F(\Xcross_n-1))$ is common to the
curves $\gamma_{n}^T$ and $\gamma_{n+1}^T$.
Equivalently, the point $(\Xcross_n,\,1+F(\Xcross_n-1))$ is common to the
curves $\gamma_{n}$ and $\gamma_{n+1}$.

\smallskip
{\rm (b)} $\Xcross_n=\Xmin_{n}=\xi_n(1)$ for $1\leq n\leq 6$
and $\Xcross_n<\xi_n(1)$ for $n\geq 7$.
\end{prop}

\begin{proof}
(a) If $x\leq\Xcross_n-1$, then by the monotonicity of
$\nu(\cdot)$ (Proposition~\ref{prop:crind}) $\nu(x)\leq n-1$, hence $F(x)=F_{\nu(x)}=F_{n-1}(x)$. 

If $x>\Xcross_n-1$, then $\nu(x)>n-1$, hence by Definition~\ref{def:crind} (of $\nu(x)$) $F_{n-1}(x)>F(x)$.

The point $(\Xcross_n-1,F(\Xcross_n-1))$ lies on the graph of the function $F(\cdot)$. In the left neighbourhood of
this point the graph of $F$ coincides with graph of $F_{n-1}$ and in the right neighbourhood -- with graph of $F_n$. Hence the curves $\gamma_{n-1}$ and $\gamma_n$ 
meet at this point.

\smallskip
(b) For $1\leq n\leq 6$ the intervals $[\xi_n(1),\xi_n(2))$
do not overlap, hence $\nu(x-1)=n-1$ for $x\leq \xi_n(1)$. 

In general, there are no points on the curve $\gamma_{n-1}[1,2]$ with abscissas greater than $\xi_{n-1}(2)$, since $\xi_{n-1}'(t)>0$ for $t\geq 2$ (Proposition~\ref{prop:Xmin}(a)). 

For $n\geq 7$, $\Xmin_n<\xi_n(1)=\xi_{n-1}(2)$.
The equality $\Xcross_n=\xi_{n}(1)$ is impossible, since
the backward branch of the curve $\gamma_n$ starting
at $(\xi_n(1),\eta_n(1))$ lies above the branch with
parameter values $t>\Xmin_n$ by Proposition~\ref{prop:pmbranches}.
\end{proof}

The graph of the function $F(x)$ is the union of the segments of the curves $\gamma_n^T$ corresponding to the parameter values $[\tau^\ell_n, \tau^r_n]\subset(\tau^0_n,1)$. We have
$$
 T_n(\tau^\ell_n)=T_{n-1}(\tau^r_{n-1})=\Xcross_n
$$ 
and 
$$
 G_n(\tau^\ell_n)=G_{n-1}(\tau^r_{n-1}).
$$
With reference to the lower envelope of the curves
$\gamma_n$, the parameter values $t_n^\ell=\tau_n^\ell+1$
and $t_n^r=\tau_n^r+1$ play the same role.

\subsection{Asymptotics of %auxiliary functions}
\texorpdfstring{$\xi_n$}{xi[n]} and 
\texorpdfstring{$\eta_n$ as $n\to\infty$}{eta[n]}}
\label{ssec:asxieta}

Put 
\begin{equation}
\label{deltan}
\delta_{n}(t)=\sum_{j=n}^\infty \frac{1}{\xi_j(t)}=\alpha_\infty(t)-\alpha_{n}(t).
\end{equation}
and define the functions $\phi(t)$ and $\psi(t)$ for $t\in[1,2]$ by
\begin{equation}
\label{phi}
\phi(t)=\sum_{j=1}^\infty\log\left(1-\frac{\delta_j(t)}{\alpha_\infty(t)}\right),
\end{equation}
\begin{equation}
\label{psi}
\psi(t)=t-\alpha_\infty(t)-\sum_{j=1}^\infty \delta_j(t).
\end{equation}

\begin{prop}
\label{prop:asphipsi}
The functions $\phi(t)$ and $\psi(t)$ are real-analytic and
the following asymptotic formulas hold:
\beq{asxin}
\log\xi_n(t)=n\log\alpha_\infty(t)+ \phi(t)+
O\left(2^{-n}\right),
\eeq
\beq{asetan}
\eta_n(t)=n\alpha_\infty(t)+ \psi(t)+
O\left(2^{-n}\right).
\eeq
\end{prop}

\begin{proof} 
The recurrence relations \eqref{recrelalxi} imply
$$
 \log\xi_n(t)=\log\xi_0(t)+\sum_{j=1}^n\log\alpha_j(t)
=\log t+\sum_{j=1}^n\log\frac{\alpha_j(t)}{\alpha_\infty(t)}
+n\log\alpha_\infty(t).
.
$$
%Since $\xi_0(t)=t$ and 
By definition, 
$\alpha_j(t)=\alpha_\infty(t)-\delta_j(t)$, so 
%we get
%$$
% \log\xi_n(t)=\log t+\sum_{j=1}^n\log\left(1-\frac{\delta_j(t)}{\alpha_\infty(t)}\right)+n\log\alpha_\infty(t).
%$$
%Hence
$$
\log\xi_n(t)-\phi(t)-n\log\alpha_\infty(t)=-
\sum_{j=n+1}^\infty\log\left(1-\frac{\delta_j(t)}{\alpha_\infty(t)}\right).
$$
Since the sequence $\xi_j(t)$ grows exponentially, taking into account the uniform estimate $\alpha_n(t)\geq 2$ from Proposition~\ref{prop:xieta-basic}(b), 
for sufficiently large $n_0$ the conditions of Lemma~
\ref{lem:est-al-xi} are met with $a=2$ and it follows by \eqref{estdeltan} that $\delta_j(t)=O(2^{-j})$, $j>n_0$. This estimate implies \eqref{asxin}.

Similarly, by the recurrence relations \eqref{recreleta}
$$
\ba{rcl}
\dst
 \eta_n(t)=\eta_0(t)+\sum_{j=0}^{n-1}\alpha_j(t)
&=&
\dst
1+n\alpha_\infty(t)-\sum_{j=0}^{n-1}\delta_j(t)
\\[2ex]
&=&
\dst
\psi(t)+n\alpha_\infty(t)+\sum_{j=n}^{\infty}\delta_j(t).
\ea
$$
The remainder term is estimated as above and we come to \eqref{asetan}.
\end{proof}

\subsection{Narrowing down the parameter domain}
\label{ssec:narrowing}

From the crude asymptotic result --- Proposition~\ref{prop:crudeas} --- we know that 
$$\frac{\eta_n(t)}{\log\xi_n(t)}=e+o(1)$$ 
on the segment of the
curve $\gamma_n$ that belongs to the lower envelope.
Hence $\log\alpha_\infty(t)/\alpha_\infty(t)=e+o(1)$
for the corresponding values of the parameter and
we conclude that $\alpha_\infty(t)$ must be close to $e$.
Therefore $t$ must be close to one of the two roots,
$t_a$ or $t_b$, of the equation $\alpha_\infty(t)=e$.

By Proposition~\ref{prop:pmbranches}, 
among the two values $t_1<t_2$ corresponding
to the same value of $\xi_n$, the inequality
$\eta_n(t_2)<\eta_n(t_1)$ takes place. Therefore
only the neighbourhood of $t_b$ is of interest for
determining the lower envelope of the curves $\gamma_n$.

Summarizing, we state

\begin{prop}
\label{prop:minint}
The values $t_n^\ell$ and $t_n^r$ defined at the end of
Sec.~\ref{ssec:special-points} tend to $t_b$ as $n\to\infty$. 
\end{prop} 

\subsection{End of proof of Theorem~\ref{thm:main}}
\label{ssec:eop-main}
Denote
$$
 \beta(t)=\frac{\alpha_\infty(t)}{\log\alpha_\infty(t)}
$$
and
$$
\zeta(t)=\psi(t)-\beta(t)\,\phi(t).
%\right)
$$
Due to Proposition~\ref{prop:minint}, in the asymptotic calculation we may choose some compact neighborhood $I\subset(t_o,2)$ 
of the point $t_b$ %(the larger root of the equation$\alpha_\infty(t)=e$) 
as the parameter interval. 

Since $\beta'(t)=(\log\alpha_\infty(t)-1)\alpha_\infty'(t)/(\log\alpha_\infty(t))^2$, the 
function $\beta(t)$ has a unique critical point (the minimum) on $I$, namely, $t=t_b$. By Proposition~\ref{prop:alpha}, $\alpha'(t_b)>0$.

The family of parametric curves $x=\log\xi_n(t)$, $y=\eta_n(t)$ is shown to satisfy all conditions of
Theorem~\ref{thm:paramcurves}. 

Using the asymptotic formulas \eqref{asxin} and \eqref{asetan}
Let us find the coefficients in the final asymptotic formula \eqref{asfabs} in our case.

%in terms of the previously introduced functions. 

The point $t_0$ in Theorem~\ref{thm:paramcurves}
corresponds to $t_b$ in this context.
The terms $r_0(t)$ and $r_1(t)$ of the general formulation
are not present in Eqs.~\eqref{asxin}--\eqref{asetan},
so $\delta(t)\equiv 0$. We have 
$a_0=\beta(t_b)=e$,
$a_1=\zeta(t_b)\approx -0.704656$. The constant $A$ in Theorem~\ref{thm:main} equals $1-\zeta(t_b)$ (remember that $G_n(t-1)=\eta_n(t)-1$ on the original extremal trajectory).

In the formula \eqref{asgenPhi}, $p_0(t_0)=\log\alpha_\infty(t_b)=1$
and $q_0(t_0)=\phi(t_b)\approx 0.6974$;
this is the numeric constant $b$ in Theorem~\ref{thm:main}.

We will derive the coefficients \eqref{coefa2a3} analytically in the concluding two lemmas, showing that $a_3=e/2$ (Lemma~\ref{lem:a3})
and $\zeta'(t_b)=0$ (Lemma~\ref{lem:Dzeta}) hence $a_2=0$.

Thus the proof of Theorem~\ref{thm:main} is complete.
\qed

\begin{lemma}
\label{lem:a3}
There holds the identity
$$
 \beta''(t_b)\left(\frac{(p_0(t_b))^2}{p_0'(t_b)}
\right)^2=e,
$$
where $p_0(t)=\log\alpha_\infty(t)$.
\end{lemma}

\begin{proof} Since $\beta=e^{p_0}/p_0$,
we have $\beta'=p_0'\cdot e^{p_0}/p_0^2\cdot (p_0-1)$.
To evaluate $\beta''(t_b)$, it suffices to differentiate
the last factor (which vanishes at $t_b$):
$$
\beta''(t_b)=\left.\frac{p_0'(t)\cdot e^{p_0(t)}}{p_0^2(t)}\cdot p_0'(t)\right|_{t=t_b}.
$$
Therefore 
$$
\beta''(t_b)\left(\frac{(p_0(t_b))^2}{p_0'(t_b)}
\right)^2=e^{p_0(t_b)}p_0^2(t_b)=e\cdot 1^2=e.
\eqno\qedhere
$$
\end{proof}

\begin{lemma}
\label{lem:Dzeta}
There holds the identity
$\zeta'(t_b)=0$, that is, $\psi'(t_b)-e\phi'(t_b)=0$.
\end{lemma}

\begin{proof}
Writing $\psi'=-\sum_{j=0}^\infty (\alpha'_\infty-\alpha'_j)$ and $\phi'=\sum_{j=1}^\infty
(\alpha'_j/\alpha_j-\alpha'_\infty/\alpha_\infty)$,
we see that
$$
\left[\psi'-e\phi'\right]_{t_b}
=
1-\alpha'_\infty(t_b)-\sum_{j=1}^\infty\left[
\alpha'_\infty\left(1-\frac{e}{\alpha_\infty}
\right)-\alpha'_j\left(1-\frac{e}{\alpha_j}
\right)\right]_{t_b}
$$
In view of the identities $\alpha_\infty(t_b)-e=0$
and $\alpha_j(t_b)-e=-\delta_j(t_b)$, the right-hand side
simplifies to 
$$
 1-\alpha'_\infty(t_b)-\sum_{j=1}^\infty\delta_j(t_b)\,\frac{\alpha'_j(t_b)}{\alpha_j(t_b)}.
$$
We will prove that in general
\begin{equation}
\label{alpha_identity}
 \sum_{j=1}^\infty \delta_j(t) \frac{\alpha_j'(t)}{\alpha_j(t)}=1-\alpha_\infty'(t).
\end{equation}

By the definition \eqref{defalpha} of $\alpha_j$,
$$
 \frac{\alpha_j'}{\alpha_j}=
\frac{\xi_j'}{\xi_{j}}-\frac{\xi_{j-1}}{\xi_{j-1}}.
$$
Applying Abel's summation-by-parts formula to the partial sum in the l.h.s. of \eqref{alpha_identity}, we get
$$
 \sum_{j=1}^N \delta_j \frac{\alpha_j'}{\alpha_j}=
\delta_N\frac{\xi_N'}{\xi_N}-\delta_1\frac{\xi_0'}{\xi_0}
+\sum_{j=1}^{N-1} (\delta_j-\delta_{j+1}) \frac{\xi_j'}{\xi_{j}}.
$$

We have $\xi_0=1$ and $\xi'_0=0$; $\xi'_N/\xi_N=O(1)$
and $\delta_N=o(1)$ as $N\to\infty$, so the boundary terms are $o(1)$.
Now, 
$\delta_{j+1}-\delta_j=\alpha_j-\alpha_{j+1}=-\xi_j^{-1}$. Therefore
$$
\sum_{j=1}^N \delta_j \frac{\alpha_j'}{\alpha_j}
=\sum_{j=1}^{N-1}(\xi_j^{-1})'+o(1)
=\alpha'_1-\alpha'_N+o(1),
$$
and the limit is $\alpha'_1-\alpha'_\infty=1-\alpha'_\infty$.
\end{proof}

\iffalse
\begin{corollary}
The critical index $\nu(x)$ satisfies the asymptotic estimate
$$
 e=\min_{1\leq t\leq 2}\log\alpha_\infty(t)\leq\liminf_{x\to\infty}\frac{\log x}{\nu(x)}
\leq\limsup_{x\to\infty}\frac{\log x}{\nu(x)} 
\leq \max_{1\leq t\leq 2}\log\alpha_\infty(t).
$$
\end{corollary}

\begin{proof}
We combine Proposition~\ref{prop:tau01} with the asymptotics of $\log\xi_n(t)$ from Lemma~\ref{lem:asphipsi}.

[Need details]

(f) [Not quite; it's too early here. Put where discuss the asymptotics of the minimizing critical point.] 

A refinement of the functional equation 
\eqref{fe1} reads
\beq{fe2}
 f(x)=\min_{0<y\leq\Xb_{\nu(x)-1}}\left(f(y)+\frac{x}{y+1}\right).
\eeq
\end{proof}
\fi

%-------------------------------------------------

\section{Constants \texorpdfstring{$A(p)$}{A(p)}: existence and computation}
\label{sec:Ap}

%\begin{proof}[Sketch of proof of Theorem~\ref{thm:genpar}]

\subsection{Proof of Theorem~\ref{thm:genpar}} 

(a) For any $n$-tuple $\mathbf{t}$, the function
$p\mapsto F^{(p)}(\mathbf{t})$ is decreasing. Hence,
if the asymptotics \eqref{Fpasym} takes place, the function
$A(p)$ is at least nondecreasing.

We will consider the cases $p>1$ and $p<1$ separately 
to justify the formula \eqref{Fpasym} and to deduce that $A(p)$ is strictly increasing. That way, a proof of (a) will be finished.

Let us note an identity that will be useful in both cases
$p>1$ and $p<1$.

For any $q>0$, putting $\tilde t_j=t_j/q$, $j=1,\dots,n$, we get
$t_j/(t_{j-1}+p)=\tilde t_j/(\tilde t_{j-1}+p/q)$.
Hence
\begin{align}
S^{(p)}_n(t_1,\dots,t_n)
&=t_1+\frac{\tilde{t}_2}{\tilde t_1+p/q}+\dots+\frac{\tilde t_n}{\tilde{t}_{n-1}+p/q}
\label{Sp-transform1}
\\[1ex]
&=t_1-\tilde{t}_1+S^{(p/q)}_n(\tilde{t}_1,\dots,\tilde {t}_n).
\label{Sp-transform}
\end{align}

\medskip
(b) Case $p>1$. Let us take $q=p$ in \eqref{Sp-transform}.
We get
$$
S^{(p)}_n(t_1,\dots,t_n)
=t_1\left(1-\frac{1}{p}\right)
+S_n(\tilde t_1,\dots,\tilde t_n)\geq 
S_n(\tilde t_1,\dots,\tilde t_n).
$$
Hence $F^{(p)}_n(x)\geq F_n(x/p)$ and $F^{(p)}(x)\geq F(x/p)$.

On the other hand, taking again $\tilde t_j=t_j/p$
and setting $t_1=0$, we obtain  
$$
\left. S^{(p)}_n(t_1,\dots,t_n)
\right|_{t_1=0}=S_{n-1}(\tilde t_1,\dots,\tilde t_{n-1}).
$$
Hence $F^{(p)}_n(x)\leq F_{n-1}(x/p)$, so $F^{(p)}(x)\leq F(x/p)$.

The identity \eqref{Fpgt1} is thus proved. The asymptotics and the expression for $A(p)$ follow
from Theorem~\ref{thm:main}.

\medskip
(c) Case $0<p<1$. Here we give only a sketch of proof, using heuristics at some steps.

The asymptotics \eqref{Fpasym} can be derived in the same way as we did it for $p=1$, by obtaining parametric
description of the curves $y=F^{(p)}_n(x)$.

To derive the functional equation, take $q=t_1+p$ in Eq.~\eqref{Sp-transform1}. Then $\tilde t_1+p/q=1$,
hence
$$
 S^{(p)}_n(t_1,\dots,t_n)=t_1+S^{(p/q)}_{n-1}(\tilde t_2,\dots,\tilde t_n).
$$
It follows that
% Check (non)strictness of inequalities in constraints!
$$
 F^{(p)}_n(x)=\inf_{q\geq p} \left(q-p+F^{(p/q)}_{n-1}(x/q)\right).
$$
Consequently,
$$
 F^{(p)}(x)=\inf_{q\geq p} \left(q-p+F^{(p/q)}(x/q)\right).
$$
From \eqref{Sp-transform} we see that 
$F^{(p)}(x)\leq F^{(p)}(x/q)$ for $q\leq 1$. As $q-p\geq 0$, the minimization range can be reduced to $q\geq 1$.
%The lower bound does not occur at $q=1$, since $1-p>0$.  

The next step needs a justification that we omit. Instead of the functions $F^{(\dots)}(\dots)$ we substitute their asymptotics
(and change the letter $q$ into $u$).
The result is 
$$
 e\log x-A(p)+o(1)=\inf_{u\geq 1} \left(u-p+e\log\frac{x}{u}-A\left(\frac{p}{u}\right)+o(1)\right).
$$
After cancellation of the $e\log x$ terms the remaining leading terms are constants (with respect to $x$). We obtain the equation \eqref{feq-Ap} with a weaker constraint: $u\geq 1$ rather that $1\leq u\leq e$.
However, the values $u>e$ can be excluded from the maximization range since the function $A(p/u)$ is nonincreasing w.r.t. $u$ and the function $e\log u-u$ strictly decreases for $u>e$.

\smallskip
Finally, let us explain the asymptotics of $A(p)$
as $p\to +0$. 
Choosing $u=e$ in the right-hand side of \eqref{feq-Ap} we
get $A(p)\geq A(p/e)+p$.
Therefore
$A(p)\geq A(pe^{-k})+p\sum_{j=0}^{k-1} e^{-j} $ for any $k\in\NN$. By monotonicity of $A(\cdot)$ we get
$A(p)\geq p(1-e^{-1})^{-1}$. Putting $k=\liminf_{p\to +0} A(p)/p$, we see that $k\geq k_0=(1-e^{-1})^{-1}$. 

Skipping a necessary justification, we assume that the $\lim_{p\to +0} A(p)/p$ exists; its value is thus $k$.
Let us show that $k=k_0$. 

Take small $\eps>0$. For sufficiently small $p$ 
and any $u\in[1,e]$ we have
$A(p/u)<(k+\eps)p/u$, so 
$$
 A(p)\leq \max_{1\leq u\leq e}\left((k+\eps)\frac{p}{u}+e\log u-u+p\right).
$$
Differentiating, we see that the maximum is attained at the point $u=u^+$ which is the larger root (close to $e$) of the quadratic equation
$$
 u^2-eu+(k+\eps)p=0.
$$
The smaller root is $u_-=e-u_+=(k+\eps)p/u_+\sim (k+\eps)p/e$. 
Since $|e\log u-u|=O(|e-u|^2)$ as $u\to e$, we have
$$ 
A(p) \leq (k+\eps)\frac{p}{u_+}+O(p^2)+p=\left(\frac{k+\eps}{e}+1\right)p+O(p^2).
$$
Therefore 
$$
 k=\lim_{p\to +0}\frac{A(p)}{p}\leq \frac{k+\eps}{e}+1,
$$
so $k\leq e/(e-1-\eps)$. Making $\eps\to 0$, we obtain
$k\leq k_0$, hence $k=k_0$.
%\end{proof}

\subsection{Tabulation of the function \texorpdfstring{$A(p)$}{A(p)}}
%\subsection{Recurrence relation}
Let us rewrite the functional equation
\eqref{feq-Ap} in the form
$$
 A(p)=\max_{p/e\leq s\leq p} \left(A(s)+\theta\left(\frac{p}{s}\right)\right)+p ,
$$
where
$$
\theta(t)=e\log t-t.
$$
Assuming that $A(\cdot)$ is differentiable, we have the
condition of extremum:
$$
 A'(s_*)-\frac{p}{s_*^2}\cdot\theta'\left(\frac{p}{s_*}\right)=0.
$$
Equivalently,
$$
  s_*^2\,A'(s_*)-es_*+p=0.
$$

We have
$$
A'(p)=1+\frac{1}{s_*}\theta' \left(\frac{p}{s_*}\right)
=1+\frac{e}{p}-\frac{1}{s_*}.
$$
%Here $p$ can be expressed in terms of $s_*$ and $A'(s_*)$,
%which we exploit shortly. 

Consider the triples $(s_*,A'(s_*),A(s_*))$
and $(p,A'(p),A(p))$ as consequtive points of the iteration process:
\begin{align}
(x_{n-1},y_{n-1},z_{n-1})&= (s_*,A'(s_*),A(s_*)),
\nonumber
\\
(x_n,y_n,z_n)&=(p,A'(p),A(p)).
\nonumber
\end{align}  
The equations defining the recurrence are
$$
 \ba{l}
\dst
 x_n=ex_{n-1}-x_{n-1}^2 y_{n-1},
\\[1ex]
\dst 
y_n=1-\frac{1}{x_{n-1}}+\frac{e}{x_n},
\\[2ex]
\dst
z_n=z_{n-1}+x_n+\theta\left(e-x_{n-1}y_{n-1}\right).
\ea
$$
Starting with some arbitrary small $x_1$ and setting
$y_1=k_0$, $z_1=k_0 x_1$, we can continue iterations
until $x_n$ exceeds $1$ for the first time. This way we can tabulate the function $A(p)$, $0<p<1$.

The actual calculation of $A(1)$ performed by this method
demonstrated an %excellent 
agreement with the %previously found 
value stated in Theorem~\ref{thm:main}, which was originally found using Theorem~\ref{thm:paramcurves}
and the functions \eqref{phi}, \eqref{psi}.
%-------------------------------------------------------

%------------------------------------------------------
\appndx{Proof of Theorem~\ref{thm:paramcurves}}{app:paramcurves}

%\begin{proof}[Proof of Theorem~\ref{thm:paramcurves}]
Note first of all that a change of order $O(n^{-2})$ in the asymptotics of $\xi_n(t)$ and $\eta_n(t)$ has an effect of order $O(u^{-2})$ on the value of $f(u)$. In particular, the regularity properties of the remainder terms (continuity, differentiability etc.) are irrelevant. Hence we may, and  will, assume that these terms are absent. It will be helpful since now the parametric equations become rational functions of $n$ and we can interpolate  to non-integer values of $n$. %cf.~Eq.~\eqref{etanu} below.

Whenever we allow non-integer values of $n$, we will use the letter $\nu$ instead.
Define 
$$
\tilde f(u)= \min_{(\nu,t):\,\xi_{\nu}(t)=u}\eta_{\nu}(t).
$$
Here the constraint is relaxed compared to that
in the definition of $f(u)$.

%Let us forget temporarily that $n$ must be integer. 
Solving the equation $u=\xi_\nu(t)$ for $\nu$, we find
that any solution (uniqueness is neither claimed nor required)
$
\nu=\nu(u,t)
$ 
has the asymptotics 
\beq{n-ut}
\nu(u,t)=\frac{u-q_0(t)-r_0(t)p_0(t) u^{-1}}{p_0(t)}+O(u^{-2})\
\quad\text{as $u\to\infty$}
\eeq
uniformly in $t\in I$. 

Put
$$
 \zeta(t)=q_1(t)-\beta(t)q_0(t).
$$
and
$$
 w(u,t)=\beta(t)u+\zeta(t)+\delta(t)u^{-1}.
$$
From \eqref{n-ut} it is easily follows that
$$
 \tilde f(u)=v_1+O(u^{-2}),
$$
where
$$
 v_1=\min_t w(u,t).
$$
Let $t_1=t_1(u)$ be the point of minimum of $w(u,t)$% 
%(for  a fixed $u$)
, so 
$v_1=w(u,t_1)$.

Due to the assumptions (i) and (iii) of Theorem, we have
$t_1=t_0+\tau$ with $\tau=O(u^{-1})$ as $u\to\infty$.
The extremal point equation $w'_t(u,t)=0$ yields
$$
\tau=-\frac{\zeta'(t_0)}{\beta''(t_0)} u^{-1}+O(u^{-2}).
$$
Substituting this value to the equation $v_1=w(u,t_0+\tau)$
we obtain
\begin{align} 
v_1&=\left(b_0+\frac{b_2}{2}\tau^2\right)u+\left(\zeta(t_0)
+\zeta'(t_0)\tau\right)+\delta(t_0)u^{-1}+O(u^{-2})
\nonumber
\\[1ex]
&=b_0u+\zeta(t_0)+\left(-\frac{(\zeta'(t_0))^2}{2b_2}+\delta(t_0)\right)u^{-1}+O(u^{-2}).
\nonumber
\\[1ex]
&=a_0u+a_1+a_2 u^{-1}+O(u^{-2}),
\nonumber
\end{align}
where the coefficients $a_0$, $a_1$, $a_2$ are those in 
Eqs.~\eqref{coefa0a1}, \eqref{coefa2a3}.

Let us now determine a correction needed to satisfy the integrality condition $n\in\NN$.
Put $t=t_1+\tau_1$.
We want to have $\nu(u,t_1+\tau_1)\in\NN$. It is easy to see that 
a change of order $O(u^{-1})$ in $t$, equivalently in $\tau_1$, causes a change of order $O(1)$ in $\nu(u,t)$.
Therefore, there is a discrete set of admissible values of $\tau_1$ separated by distances of order $O(u^{-1})$. 

We have
$$
 w(u,t)-v_1=\frac{1}{2}\left.\frac{\partial^2 w}{\partial t^2}\right|_{t=t_1}\tau_1^2+O(u^{-2})=\frac{b_2}{2}u\tau_1^2
+O(u^{-2}),
$$
because $w'_t|_{t=t_1}=0$ and 
$w''_{tt}|_{t=t_1}=b_2+O(u^{-1})$.
The minimum value of the right-hand side, %in \eqref{wut},
up to an error $O(u^{-2})$, corresponds to the minimum value of $|\tau_1|$, up to an error $O(u^{-2})$, in the above mentioned discrete set.

From Eq.~\eqref{n-ut} we infer (since $\nu'_t(u,t)=-u(1/p_0(t))'+O(1)$): 
$$
 n=\nu(u,t_1)-\frac{u p_0'(t_1)}{p_0(t_1)^2}\tau_1
 +O(u^{-1}).
$$
Hence the minimum admissible value of $|\tau_1|$ is
$$
 |\tau_1|=\ab{\nu(u,t_1)+O(u^{-1})}\cdot \frac{p_0(t_1)^2 }{ p_0'(t_1)}u^{-1}.
$$
Due to the continuity of the function $\ab{\cdot}$
and the estimate $t_1-t_0=O(u^{-1})$ we can write
$$
|\tau_1|=\ab{\frac{u-q_0(t_0)}{p_0(t_0)}}\cdot \frac{p_0(t_0)^2 }{ p_0'(t_0)} u^{-1}
+O(u^{-2}).
$$
The final substitution into $f(u)=v_1+\frac{b_2}{2}u\tau_1^2+O(u^{-2})$ leads to the 
complete asymptotic formula \eqref{asfabs} and the expression for $a_3$ as in \eqref{coefa2a3}.
%\end{proof}

\appndx{Proof of Proposition~\ref{prop:alpha}}{app:alpha}

%\begin{proof}[Proof of Theorem~\ref{thm:prop-alpha}]
The proof is based on a series of lemmas that are placed at the end of this Appendix.
Some of the qualitative results depend on numerical evaluations. Speaking on the matter of rigor, one might refer to interval analysis; however, the accuracy of approximation is never a critical issue and we work with rational functions of relatively low complexity. So
% at that stage 
when needed
we just refer to 
the computed data truncated to low digits without being too pedantic.

%The high-precision evaluations \eqref{num-alpha} 
%need a bit more care: we need to examine numerical stability of the recurrent procedure since 
%the required number of iterations is about 40. We say more about it in Appendix~A.

We will use the notation \eqref{deltan} and
\begin{equation}
\label{deltamn}
\delta_{n,m}(t)=\sum_{j=n}^{m-1} \frac{1}{\xi_j(t)}=\alpha_m(t)-\alpha_{n}(t)
\quad (n<m).
\end{equation}

%\smallskip
1. In Lemmas~\ref{lem:est-al-xi} through \ref{lem:alphaD}
we establish bounds, similar in nature to the simple estimates $2\leq\alpha_n(t)<3$ ($t\geq 1$) of Proposition~\ref{prop:xieta-basic}(b), here --- for $\alpha_n(z)$ and $\xi_n(z)$ with $z\in\Gamma$,
where $\Gamma$ is a certain contour in the complex 
plane whose interior contains the segment $[1,2]$.
It follows that the sequence of analytic functions
$\alpha_n(z)$ converges uniformly on $\Gamma$ and therefore its limit $\alpha_\infty(z)$ is analytic in the interior of $\Gamma$. In particular, $\alpha_\infty(t)$ is real analytic for $t\in[1,2]$.

\smallskip
2. We claim that $\alpha_n''(t)>0$ in $[1,2]$ for $n\in\NN\cup\{\infty\}$. Indeed,
$\alpha_1''(t)=2t^{-3}\geq 1/4$. For $n=2,3,4$ the claimed inequality is acsertained by the numerical minimization of the corresponding rational function.

Put
$$
 A=\min_{t\in[1,2]} t^3\alpha''_4(t).
$$
We have the numerical fact
\beq{estddalpha4}
 A>2.32.
\eeq
Now, 
$$
 \alpha_n(t)''=\alpha_4''(t)+\delta_{4,n}''(t)
$$
By Cauchy's formula, 
$$
 \delta_{4,n}''(t)=\frac{2}{2\pi i}\int_{\Gamma}
\frac{\delta_{4,n}(z)}{(z-t)^3}\,dz.
$$
Due to the choice of the contour $\Gamma$ we have
%(see Lemma~\ref{lem:}) 
$|z-t|\geq t/\sqrt{2}$ for any $t\in[1,2]$ and $z\in\Gamma$.
Therefore
$$
 |\delta_{4,n}''(t)|\leq \frac{2^{3/2}}{\pi t^3}\int_{\Gamma}|\delta_{4,n}(z)|\,|dz|.
$$
Put
\beq{deltamax}
\delta_n^*(z)=\sum_{j=n}^\infty\frac{1}{|\xi_j(z)|}.
\eeq
Clearly, $|\delta_{4,n}(z)|<\delta_4^*(z)$ for all $n>4$ (see \eqref{deltan}).

Define
\beq{deltaB}
 B=\int_{\Gamma}\delta_4^*(z)\,|dz|.
\eeq
For any $t\in[1,2]$ we have
$$
\alpha''_4(t)-|\delta''_4(t)|>\alpha''_4(t)
-\frac{2^{3/2}}{\pi t^3}B\geq \frac{1}{t^3}
\left(A-\frac{2^{3/2}}{\pi}B\right)
$$
Numerically, $B<2.48$ by Lemma~\ref{lem:estB}. Hence
$$
 A-\frac{2^{3/2}}{\pi}B>2.32-\frac{2^{3/2}}{\pi}\cdot 2.48> 0.08>0.
$$
The inequality $\alpha_n''(t)>0$ is proved; the limit case $\alpha_\infty''(t)>0$ also follows.

\smallskip
3. We have, numerically,
$$
 \alpha_4'(1)<-0.7,
\qquad
\alpha_4'(2)>0.48.
$$
Now, for $t\in[1,2]$,
$$
 |\delta_{4,n}'(t)|=\frac{1}{2\pi}\left|\int_{\Gamma}
\frac{\delta_{4,n}(z)}{(z-t)^2}\,dz\right|
<\frac{2B}{2\pi t^2}.
$$
Using again the estimate $B<2.48$ of Lemma~\ref{lem:estB}, we get 
and $|\delta_{4,n}'(2)|<0.2$, hence $\alpha_{n}'(2)=\alpha_{4}'(2)+\delta_{4,n}'(2)>0.28>0$
and $\alpha_\infty'(2)>0$. 

The same method yields $|\delta_{4,n}'(1)|<0.79$,
which is not enough to claim that
$\alpha_{n}'(1)<0$.
%, since $0.7<0.79$. 
We refer to the estimate for
\beq{B1}
 B_1=\int_{\Gamma}
\frac{\delta^*_{4}(z)}{|z-1|^2}\,|dz|
\eeq 
provided by the same Lemma~\ref{lem:estB}:
$B_1<4.2$. We get $|\delta'_{4,n}(1)|<B_1/(2\pi)<0.67$. Since $0.67<0.7$,
the estimate $\delta'_{4,n}(1)<0$ follows.

Thus, for any $n\geq 4$, the function $\delta'_{4,n}(t)$ is increasing in $[1,2]$ and changes sign. Hence it has the unique root $t_{n,o}$.
The same is true for $n=\infty$. 

\smallskip
4. We have $\alpha_2(2)=2+1/2+1/5=2.7$; $\xi_2(2)=27/2$; $\alpha_3(2)=27/10+2/27=749/270>2.774$, hence $\alpha_n(2)>e$ for all $n\geq 3$. 

\smallskip
5. To prove the inequality $\alpha(t_o)<e$ we use the evaluations
$$
\ba{l}
\xi_2(3/2)=193/24>8. %\\
\alpha_3(3/2)\in (2.5,2.6).  
\ea
$$
It follows by \eqref{recrelalxi} that
$\alpha_\infty(t_o)\leq \alpha_\infty(3/2)<2.6+1/12<e$. 

The proof of Proposition~\ref{prop:alpha} is complete.
\qed

%\medskip
%On the stability of recurrent procedure
%
%...........

%%
\begin{figure}[bht]
\begin{minipage}{0.95\textwidth}
{\maplecodestyle
\begin{lstlisting}
 > AlphaXi:=proc(t,n)    # Procedure to compute [alpha_n(t),xi_n(t)]
   local a,x,i;  a:=t; x:=t;
   for i from 1 to n do
      a:=a+1./x; 
      x:=x*a;
   od; 
   return [a,x]; end proc:
 > AlphaXi(t,1);  # test: symbolic computation of [alpha_1(t),xi_1(t)] 
\end{lstlisting}
}
\centerline{\mapleoutpstyle $
 \left[t+\frac{1}{t}, t\left(t+\frac{1}{t}\right)\right]
$}
{\maplecodestyle
\begin{lstlisting}
 > Optimization[Minimize](abs(AlphaXi(3.5*exp(I*phi),4)[1]),phi=0..Pi/4);
\end{lstlisting}
}
\centerline{\mapleoutpstyle $
 [3.43600269664558, [phi = .785398163397448]]
$}
{\maplecodestyle
\begin{lstlisting}
 > Optimization[Minimize](abs(AlphaXi(3.5*exp(I*phi),4)[2]),phi=0..Pi/4);
\end{lstlisting}
}
\centerline{\mapleoutpstyle $
 [503.687917872391, [phi = .785398163397448]]
$}
{\maplecodestyle
\begin{lstlisting}
 > sqrt(3.43600269664558); # min ((alpha_4(z))^(1/2)), z in Gamma_3
\end{lstlisting}
}
\centerline{\mapleoutpstyle $
 1.853645785
$}
{\maplecodestyle
\begin{lstlisting}
 > 1/sqrt(503.687917872391); # max((xi_4(z))^(-1/2)), z in Gamma_3
\end{lstlisting}
}
\centerline{\mapleoutpstyle $
 0.04455733765
$}
\end{minipage}

\begin{picture}(380,5)
\put(0,0){\framebox(380,250){}}
\end{picture}
\caption{A fragment of a computer algebra system (Maple) worksheet used in the computations for Proposition~\ref{prop:alpha} and Lemma~\ref{lem:raw-est-contour}}
\label{fig:CAScode}
\end{figure}

\begin{figure}[thb]
\begin{center}
\definecolor{wwwwww}{rgb}{0.4,0.4,0.4}
\definecolor{zzttqq}{rgb}{0.6,0.2,0}
\definecolor{xdxdff}{rgb}{0.49019607843137253,0.49019607843137253,1}
\begin{tikzpicture}[line cap=round,line join=round,>=triangle 45,x=0.15cm,y=0.15cm]
\clip(-5,-20) rectangle (33,20);
% segment [1,2]:
\draw [line width=2pt,color=zzttqq] (8,0)-- (16,0);
% \Gamma_1:
\draw [shift={(0.0730113,0.085204)},line width=1pt]  plot[domain=-0.84244:0.781975,variable=\t]({1*1.928890898211995*cos(\t r)+0*1.928890898211995*sin(\t r)},{0*1.928890898211995*cos(\t r)+1*1.928890898211995*sin(\t r)});
%
% \Gamma_3:
\draw [shift={(-0.1149326,0.026446)},line width=1pt]  plot[domain=-0.782187:0.7828295,variable=\t]({1*28.11494503749307*cos(\t r)+0*28.11494503749307*sin(\t r)},{0*28.11494503749307*cos(\t r)+1*28.11494503749307*sin(\t r)});
%
% \Gamma_2:
\draw [line width=1pt] (1.42,-1.42) -- (6.80,-6.80);
\draw [->,line width=1pt] (13.02,-13.02) -- (19.82,-19.82);
%
% \Gamma_4:
\draw [->,line width=1pt] (19.82,19.82) -- (1.42,1.42);
%
% \Gamma_{20}:
\draw [line width=2pt,dashed] (6.80,-6.80) -- (13.02,-13.02);
%
% Segment [t,z]
\draw [line width=1.2pt,dash pattern=on 2pt off 2pt,color=wwwwww] (7.82,7.84)-- (11.56,0);
%
% \arrows:
%\draw [->,line width=0.3pt] (1.98,0.2) -- (1.98,-0.65);
\draw [->,line width=1pt] (27.98,-0.5) -- (27.98,2);
%-------
\draw [fill=xdxdff] (1.356,-1.35) circle (0.5pt);
\draw (-0.7,0.45) node {$\Gamma_1$};
\draw [fill=xdxdff] (1.44,1.44) circle (0.5pt);
\draw [fill=zzttqq] (8,0) circle (2.5pt);
\draw (8,-2.3) node {$1$};
\draw [fill=zzttqq] (16,0) circle (2.5pt);
\draw (16,-2.3) node {$2$};
\draw [fill=xdxdff] (19.83,-19.79) circle (0.5pt);
\draw [fill=xdxdff] (19.82,19.86) circle (0.5pt);
\draw (14,-17.5) node {$\Gamma_{2}$};
\draw (12,-8) node {$\Gamma_{20}$};
\draw (31,2) node {$\Gamma_3$};
\draw (11,14.5) node {$\Gamma_4$};
\draw [fill=black] (11.56,0) circle (2pt);
\draw (12.3,2) node {$t$};
\draw [fill=black] (7.82,7.84) circle (2pt);
\draw (7,9.8) node {$z$};
\end{tikzpicture}
\end{center}
\caption{The contour $\Gamma$}
\label{fig:contour_gamma}
\end{figure}
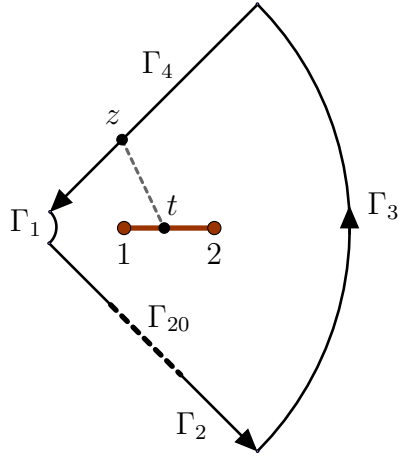

\medskip
Here are the lemmas
used in the proof of Proposition~\ref{prop:alpha}.

\begin{lemma}
\label{lem:est-al-xi}
Let us fix $n_0\in\ZZ_+$.
Suppose that 
\beq{cond_for_alxi_estimate}
|\alpha_{n_0}|^{1/2}-|\xi_{n_0}|^{-1/2}>1
\eeq
and $a>1$ is such that
$$
1+|\xi_{n_0}|^{-1/2}\leq a< |\alpha_{n_0}|^{1/2}
$$
Then the inequalities $|\alpha_{n}|>a$ and 
$|\xi_{n}|\geq|\xi_{n_0}|a^{n-n_0}$ hold true for all $n\geq n_0$.
Consequently,
$$
 \alpha_\infty=\lim_{n\to\infty} \alpha_n
$$ 
exists and $|\alpha_\infty|\geq a$.

The differences $\delta_{n,m}$ defined in \eqref{deltamn} satisfy the estimate
\beq{estdeltan}
|\delta_{n,m}|\leq|\xi_{n_0}|^{-1}\frac{a^{n_0-n}}{1-a^{-1}}.
%\leq (a+1)a^{n_0-n+1}.
\eeq
The same upper bound holds for $\delta_n$ defined in \eqref{deltan}.
\end{lemma}

\begin{proof}
The inequalities $|\alpha_{n}|>a$ and 
$|\xi_{n}|\geq|\xi_{n_0}|a^{n_0-n}$ follow by induction as soon as
we show that
$|\alpha_{n_0}|-\sum_{j=0}^\infty |\xi_{n_0}|^{-1} a^{-j}\geq a$.
And this is true, since
$$
|\alpha_{n_0}|-\sum_{j=0}^\infty |\xi_{n_0}|^{-1} a^{-j}\geq a^2-\frac{(a-1)^2}{1-a^{-1}}=a.
$$

The estimate \eqref{estdeltan} follows by a similar calculation.
\end{proof}

\begin{lemma}
\label{lem:raw-est-contour}
Consider the contour in the complex plane $\Gamma=\Gamma_1\cup\Gamma_2\cup\Gamma_3\cup\Gamma_4$, 
where $\Gamma_1$ and $\Gamma_3$ are the quarter-circle arcs (in polar coordinates
$z=re^{i\phi}$) 
$\Gamma_1$: $r=0.25$, $|\phi|<\pi/4$,
$\Gamma_3$: $r=3.5$, $|\phi|<\pi/4$,
$\Gamma_2$ is the linear segment
$0.25\leq r\leq 3.5$, $\phi=-\pi/4$, and
$\Gamma_4$ is complex-conjugate to $\Gamma_2$. {\rm (Fig.~\ref{fig:contour_gamma})}

%For any $t\in[1,2]$ and any $z\in\Gamma$ we have
%$$
% |z-t|\geq \frac{t}{\sqrt{2}}.
%$$

Furthermore, define the subset of $\Gamma_2$,
$$
 \Gamma_{20}=\Gamma_2\cap \{z\,:\,1.15\leq |z|\leq 2.3\}
$$
and the corresponding subset $\Gamma_{40}\subset\Gamma_4$. 

The following numerical estimates hold true:

$$
\ba{lll}
z\in\Gamma_1:\quad &
 |\alpha_4(z)|^{1/2}>2.2, & \quad |\xi_4(z)|^{-1/2}<0.1; \\[1ex]
z\in\Gamma_3:\quad &
 |\alpha_4(z)|^{1/2}>1.85, & \quad |\xi_4(z)|^{-1/2}<0.05; \\[1ex]
z\in\Gamma_{20}:\quad &
 |\alpha_4(z)|^{1/2}>1.31, & \quad |\xi_4(z)|^{-1/2}<0.28; 
\\[1ex]
z\in\Gamma_2\setminus\Gamma_{20}:\quad &
 |\alpha_4(z)|^{1/2}>1.47, & \quad |\xi_4(z)|^{-1/2}<0.27.
\ea
$$
On $\Gamma_{40}$ and $\Gamma_4\setminus\Gamma_{40}$
the estimates are the same as on $\Gamma_{20}$ and $\Gamma_2\setminus\Gamma_{20}$.
\end{lemma}

\begin{proof}[Proof:] computation. (E.g.: compare the numbers in the line ``$z\in\Gamma_3$'' above and in the last two lines of the listing, Fig.~\ref{fig:CAScode}.)
\end{proof}

\begin{lemma}
\label{lem:alphaD}
The functional sequence $\{\alpha_n(z)\}$ converges uniformly in the region $D=\{z\mid 0.25\leq|z|\leq 3.5,\;|\arg z|\leq\pi/4\}$, whose boundary is $\Gamma$,
and its limit $\alpha_\infty(z)$ is analytic in the interior of $D$.
\end{lemma}

\begin{proof}
The 
condition \eqref{cond_for_alxi_estimate}
is fulfilled on $\Gamma$ with $n_0=4$ according to
Lemma \ref{lem:raw-est-contour}. (We can even take the common value $a=0.29$ for all $z\in\Gamma$.)
The uniform convergence for $z\in\Gamma$ follows by Lemma~\ref{lem:est-al-xi} and for all $z\in D$.

To conclude that the sequence $\{\alpha_n(z)\}$ uniformly converges in $D$, it remains to show (due to the maximum principle) that the rational function $\alpha_n(z)$ is analytic for each $n\geq 4$ inside $\Gamma$. %, i.e.\ that it does not have poles.

A straightforward though a bit tedious 
check shows that the denominator of 
%the rational function 
$\alpha_4(z)$, equal to 
$$
z(z^2+1)(z^4+2z^2+z+1)(z^8+4z^6+2z^5+7z^4+4z^3+6z^2+2z+1),
$$
does not have roots in $D$. (The index of each polynomial factor along $\Gamma$ equals $0$.)
For the same reason, $\xi_3(z)$ does not have poles and zeros in $D$. 

Consider the induction hypothesis: $\alpha_n(z)$ and $\xi_{n-1}(z)$ do not have zeros and poles in $D$.  
The induction step $n\to n+1$ goes as follows.
By \eqref{recrelalxi}, $\xi_{n}(z)$ does not have zeros and poles in $D$, then $\alpha_{n+1}(z)$ does not have poles in $D$. Finally, 
the proof of Lemma~\ref{lem:est-al-xi} shows that
$|\delta_{4,n}(z)|<|\alpha_4(z)|$ on $\Gamma$ for any $n>4$. By Rouch\'{e}'s theorem,
$\alpha_n(z)$ also does not have zeros in $D$.
\end{proof}

For the purpose of integration in the proof of Proposition~\ref{prop:alpha} we assume 
that the contour $\Gamma$ is positively (i.e.\ counterclockwise) oriented. The estimates for complex integrals are based on the following bound for a real integral.

\begin{lemma}
\label{lem:estB}
The quantity $B$ defined in 
\eqref{deltamax}--\eqref{deltaB} is estimated as
$$
 B<2.48.
$$
The quantity $B_1$ defined in \eqref{B1} is estimated as
$$
 B_1<4.2.
$$
\end{lemma}

\begin{proof}
We refer to the numerical data of Lemma~\ref{lem:raw-est-contour}.
For $z\in\Gamma_1$ we take $a=2.2$; then by
the estimate \eqref{estdeltan} of Lemma~\ref{lem:est-al-xi} $\delta_4^*(z)<0.1^2/(1-2.2^{-1})<0.02$
and
$$
 \int_{\Gamma_1}\delta_4^*(z)\,|dz|<\frac{\pi}{2}\cdot 0.25\cdot 0.02<0.008.
$$
For $z\in\Gamma_3$ we put $a=1.8$ and obtain
$$
 \int_{\Gamma_3}\delta_4^*(z)\,|dz|<\frac{\pi}{2}\cdot 3.5\cdot \frac{0.05^2}{1-1.85^{-1}}<0.03.
$$ 
For $z\in\Gamma_{20}$ we put $a=1.3$ and obtain
$$
 \int_{\Gamma_{20}}\delta_4^*(z)\,|dz|<(2.3-1.15)\sqrt{2}\cdot \frac{0.28^2}{1-1.31^{-1}}<0.54.
$$ 
For $z\in\Gamma_{2}\setminus\Gamma_{20}$ we put $a=1.47$ and obtain
$$
 \int_{\Gamma_{2}\setminus\Gamma_{20}}\delta_4^*(z)\,|dz|<(1.15-0.25+3.5-2.3)\sqrt{2}\cdot \frac{0.27^2}{1-1.47^{-1}}<0.68.
$$ 
The latter two estimates apply also to $\Gamma_{40}$
and $\Gamma_{4}\setminus\Gamma_{40}$ respectively.
Therefore
$$
\int_\Gamma\delta_4^*(z)\,|dz|<0.008+0.03+2\cdot(0.54+0.68)
<2.48.
$$

To obtain the estimate for $B_1$, we take into account the inequality $|z-1|^2>1/2$, $z\in\Gamma$, 
and its (very crude) improvement $|z-1|>1$ on the part $\{z\mid |z|>\sqrt{2}$ of $\Gamma$.
In fact, we employ the latter improvement just on the corresponding parts of $\Gamma_2\setminus\Gamma_{20}$
and $\Gamma_4\setminus\Gamma_{40}$. 
We have
$$
 \int_{\Gamma_2\setminus\Gamma_{20}}
 \frac{\delta_4^*(z)}{|z-1|^2}\,|dz|<
 \left(\frac{1.15-0.25}{1/2}+\frac{3.5-2.3}{1}\right)
\sqrt{2}\cdot\frac{0.27^2}{1-1.47^{-1}}<0.97.
$$
We arrive at the claimed estimate:
$$
 B_1<2(0.008+0.03+2\cdot 0.54)+2\cdot 0.97 <4.2.
\eqno\qedhere
$$
\end{proof}

\end{document}